\documentclass[11pt, letterpaper]{article}
\usepackage{amsthm,mathrsfs,fancyhdr,subfigure,amssymb,amsmath}
\usepackage{graphicx} 
\usepackage[noadjust,sort]{cite}
\usepackage{authblk}
\usepackage[normalem]{ulem}

\theoremstyle{plain} 

\usepackage{rotate}
\usepackage{needspace}
\usepackage[margin=1.05in]{geometry}
\usepackage{rotate}

\usepackage{doi}

\newtheorem{theorem}{Theorem}[section]
\newtheorem*{theorem*}{Theorem}
 
\newtheorem*{definition*}{Definition} 
\newtheorem{corollary}[theorem]{Corollary}
\newtheorem{remark}[theorem]{Remark}
\newtheorem{example}[theorem]{Example}
\newtheorem{conjecture}[theorem]{Conjecture}
\newtheorem{lemma}[theorem]{Lemma}
\newtheorem{proposition}[theorem]{Proposition}
\newtheorem{question}[theorem]{Question}

\numberwithin{equation}{section}
\numberwithin{figure}{section}

\newcommand{\E}{{\mathbb E}}
\newcommand{\R}{{\mathbb R}}

\usepackage{amsmath, tikz, amssymb, color}

\usepackage{caption}

\usepackage{algorithm}
\usepackage{algpseudocode}

\newcommand{\Gnp}{G_{n,p}}

\newcommand{\Gp}{G_p}

\newcommand{\ds}{{\rm vol}}
\newcommand{\vol}{{\rm vol}}
\newcommand{\Bin}{{\rm Bin}}

\newcommand{\eps}{\varepsilon}

\renewcommand{\AA}{ \mathcal{A} }
\newcommand{\A}{A}
\newcommand{\cB}{\mathcal{B}}

\newcommand{\cQ}{\mathcal{Q}}
\newcommand{\q}{q^*}
\newcommand{\qk}{q^*_{\leq k}}
\newcommand{\qtwo}{q^*_{\leq 2}}

\newcommand{\inter}{{\rm int}}

\newcommand{\xvec}{\boldsymbol{x}}
\newcommand{\yvec}{\boldsymbol{y}}

\newcommand{\Abs}[1]{\bigl|#1\bigr|}

\numberwithin{equation}{section}

\usepackage{rotating}

\newcommand{\NN}{{\mathbb N}}
\newcommand{\pr}{{\mathbb P}}
\renewcommand{\E}{{\mathbb E}}
\newcommand{\var}{{\rm var}}

\renewcommand{\Bin}{{\rm Bin}}
\renewcommand{\ds}{{\rm vol}}
\newcommand{\cA}{{\mathcal A}}
\newcommand{\cE}{{\mathcal E}}

\newcommand{\TT}{\mathcal{T}}
\renewcommand{\eps}{\varepsilon}
\newcommand{\eint}{e_{\cA}^{\rm int}}

\newcommand{\indic}{{\mathbf 1}}

\date{ \vspace{-10mm}}

\title{Modularity and partially observed graphs.}

\author[1]{Colin McDiarmid\thanks{Email: \textit{cmcd@stats.ox.ac.uk}.}}
\author[2]{Fiona Skerman\thanks{Email: \textit{fiona.skerman@math.uu.se}. Partially supported by the Wallenberg AI, Autonomous Systems and Software Program (WASP) and the project AI4Research at Uppsala University. Part of this work was done while visiting the Simons Institute for the Theory of Computing, supported by a Simons-Berkeley Research Fellowship.}}
\affil[1]{Department of Statistics, University of Oxford.}
\affil[2]{Department of Mathematics, Uppsala University.}

\begin{document}

\maketitle

\begin{abstract}
Suppose that there is an unknown underlying graph $G$ on a large vertex set, and we can test only a proportion of the possible edges to check whether they are present in $G$.  If $G$ has high modularity, is the observed graph $G'$ likely to have high modularity?  
We see that this is indeed the case under a mild condition, in a natural model where we test edges at random.  We find that  $\q(G') \geq \q(G)-\eps$ with probability at least $1-\eps$, as long as the expected number edges in $G'$ is large enough. Similarly, $\q(G') \leq \q(G)+\eps$ with probability at least $1-\eps$, under the stronger condition that the expected average degree in $G'$ is large enough. Further, under this stronger condition, finding a good partition for $G'$ helps us to find a good partition for $G$.

We also consider the vertex sampling model for partially observing the underlying graph: we find that for dense underlying graphs we may estimate the modularity by sampling constantly many vertices and observing the corresponding induced subgraph, but this does not hold for underlying graphs with a subquadratic number of edges. Finally we deduce some related results, for example showing that under-sampling tends to lead to overestimation of modularity.
\end{abstract}


\section{Introduction and main results}
\label{sec.intro}

The modularity of a graph is a measure of the extent to which the graph breaks into separate communities.  For a given graph $G$, each partition $\cA$ of the vertices has a modularity score $q_{\cA}(G)$, with higher values indicating that the partition better captures community structure in $G$. The modularity $\q(G)$ of the graph $G$ is defined to be the maximum over all vertex partitions of the modularity score, and satisfies $0\leq \q(G) < 1$.
See Section~\ref{subsec.mod} for definitions and some background.

Suppose that there is an unknown underlying graph $G$ on a large given vertex set, and we can test only a small proportion of the possible edges to check whether they are present in $G$, or perhaps we can test many possible edges but there is a chance that we fail to spot an edge. We assume that there are no false positives, where there is no edge in $G$ but we think there is one, except briefly in the concluding remarks. We consider two questions. (a) If $G$ has high modularity, is the observed graph $G'$ likely to have high modularity?  In other words, if $G'$ has low modularity can we assert with some confidence that $G$ has low modularity?  (b)~Conversely, if $G$ has low modularity, is $G'$ likely to have low modularity?
We find that, in a natural model where we test possible edges independently at random the answer to (a) is yes, under the condition that the expected number of edges found is large enough; and the answer to (b) is yes, under the stronger condition that the expected average degree in $G'$ is large enough.

To investigate these questions we use the random graph models $G_p$ and $G_m$ described in subsections~\ref{subsec.modGp} and~\ref{subsec.modGm} below. First we recall the definition of modularity.

\needspace{3\baselineskip}
\subsection{Modularity : definition and notation}
\label{subsec.mod}
Given a graph $G=(V,E)$, modularity gives a score to each partition of the vertex set $V$; and the (maximum) modularity $\q(G)$ of $G$ is the maximum of these scores over all vertex partitions. Let~${\mathbf 1}_{uv\in E}$ be the indicator that $uv$ is an edge.  For a set $A$ of vertices, let $e(A)$ denote the number of edges within $A$, and let $\vol(A)$ denote the sum of the degree $d_v$ over the vertices $v$ in $A$. 
\begin{definition*}[\cite{NewmanGirvan}, see also~\cite{NewmanBook}]\label{def.mod}
Let $G$ be a graph with $m\geq 1$ edges. For a vertex partition $\AA$ of~$G$, the modularity score of $\cA$ on $G$ is 
\begin{eqnarray*}
 q_\AA(G) =  
\frac{1}{2m}\sum_{\A\in \AA} \sum_{u,v \in A} 
\left( {\mathbf 1}_{uv\in E} - \frac{d_u d_v}{2m} \right)
 = 
\frac{1}{m}\sum_{\A \in \AA} e(\A) - \frac{1}{4m^2}\sum_{\A\in \AA} \ds (\A)^2\,.
\end{eqnarray*}
The modularity of $G$ is $\q(G)=\max_\AA q_\AA(G)$, where the maximum is over all vertex partitions~$\cA$. (For an empty graph $G$ (i.e. with no edges, $m=0$): we set $q_{\cA}(G)=0$ for each vertex partition $\cA$, and $\q(G)=0$.) 
\end{definition*}
Directly from the definition we have $0 \leq \q(G) <1$ for all graphs. Note isolated vertices are irrelevant - they are not counted in the formula for the modularity score.
The second expression for $q_{\cA}(G)$ expresses it as the difference of the \emph{edge contribution} or \emph{coverage} $q^E_\AA(G)=\tfrac{1}{m}\sum_A e(A)$, and the \emph{degree tax} $q^D_\AA(G)=\tfrac{1}{4m^2}\sum_A\vol(A)^2$.

\needspace{3\baselineskip}
\subsection{Modularity of the random graph $G_p$ obtained by edge-sampling}
\label{subsec.modGp}
Given a graph $G$ and $0<p \leq 1$, let $G_p$ be the random subgraph of~$G$ on the same vertex set obtained by considering each edge of $G$ independently and keeping it in the graph with probability $p$ (and otherwise deleting it).  Thus the binomial or Erd\H{o}s-R\'enyi random graph $\Gnp$ may be written as $(K_n)_p$, where $K_n$ is the $n$-vertex complete graph.  Let $e(H)$ denote the number of edges in a graph $H$: thus the expected number of edges in $G_p$ is $e(G) p$. The first of our two theorems concerns when we want to have $\q(G_p) > \q(G)-\eps$ with probability near 1.
\begin{theorem}\label{thm.obsmod}
Given $\eps>0$ there exists $c=c(\eps)$ such that the following holds. For each graph~$G$ and probability $p$ such that $e(G)p \geq c$, the random graph $G_p$ satisfies $\q(G_p) > \q(G)-\eps$ with probability $\geq 1-\eps$.
\end{theorem}
The proof of Theorem~\ref{thm.obsmod} (in Section~\ref{sec.proofthm1}) will show that we may take $c(\eps)=\Theta(\eps^{-3}\log \eps^{-1})$. 
Following the proof, we give an example which shows that $c(\eps)$ must be at least $\Omega(\eps^{-1} \log \eps^{-1})$. See Figure~\ref{fig.sims} on page~\pageref{fig.sims} for simulations illustrating Theorem~\ref{thm.obsmod} (and Theorem~\ref{thm.moddiff}), and see Figure~\ref{fig.simsBIG} on~page~\pageref{fig.simsBIG} for simulations run on a larger underlying graph.

The second of our theorems concerns when we want also to have $\q(G_p) < \q(G)+\eps$ with probability near 1, and it shows that this will happen if the expected average degree is sufficiently large.  
We see also that, in this case, finding a good partition $\cA$ for $G_p$ helps us to find a good partition $\cA'$ for $G$. Observe that the expected average degree in~$G_p$ is $2 e(G)p/v(G)$, where $v(G)$ is the number of vertices in $G$.

\begin{theorem}\label{thm.moddiff}
For each $\eps>0$, there is a $c=c(\eps)$ such that the following holds.  Let the graph $G$ and probability $p$ satisfy $e(G)p/v(G) \geq c$.  Then, with probability $\geq 1-\eps$ the following statements (a) and (b) hold:
\begin{itemize}
 \item[(a)] the random graph $G_p$ satisfies $|\q(G_p) - \q(G)| < \eps$; and
 \item[(b)] given any partition $\cA$ of the vertex set, in a linear number of operations (seeing only~$G_p$) the greedy amalgamating algorithm finds a partition $\cA'$ with $ q_{\cA'}(G) \geq q_{\cA}(G_p) - \eps$.
\end{itemize}
\end{theorem}
Part (b) says roughly that, given a good partition $\cA$ for $G_p$, we can quickly construct a good partition $\cA'$ for $G$. 
Using also part (a), we may see that, if the partition $\cA$ for $G_p$ satisfies $q_\cA(G_p) \geq \q(G_p) -\eps$ with probability at least $1-\eps$, then the partition $\cA'$ satisfies $q_{\cA'}(G) > \q(G)-2\eps$ with probability at least $1-2\eps$. See Section~\ref{sec.fat} for the greedy amalgamating algorithm used to construct~$\cA'$.
The proof of Theorem~\ref{thm.moddiff} (in Section~\ref{sec.proofthm2}) will show that we may take $c(\eps)=\Theta(\eps^{-3} \log \eps^{-1})$. Following the proof, we give an example which shows that $c(\eps)$ must be at least~$\Omega(\eps^{-2})$.

The assumption in Theorem~\ref{thm.moddiff} that the expected average degree is large is of course much stronger than the assumption in Theorem~\ref{thm.obsmod}, but still $e(G_p)$ may be much smaller than $e(G)$. If we go much further, and assume that at most an $\eps$-proportion of edges are missed, that is $|E \setminus E'| \leq \eps |E|$, then deterministically we have $|\q(G')-\q(G)| \leq 2 \eps$ by~\cite{ERmod}, see Section~\ref{subsec.robustness}.

\needspace{3\baselineskip}
\subsection{Modularity of the random graph $G_m$ obtained by limited search}
\label{subsec.modGm}

Let $G_m$ be the graph on vertex set $[n]$ with edge set a uniformly random $m$-edge subset of $E(G)$. The graph $G_m$ may also be sampled by a randomised procedure to find edges of~$G$, where we stop once we find `enough' edges:  we are given $n$, there is an unknown graph $G$ on $V= [n]$, and we query possible edges of $G$, uniformly at random, until we find $m$ edges. 
\begin{corollary} \label{cor.obsmod}
Given $\eps>0$ there exists $m_0=m_0(\eps)$ such that the following holds. For all $m \geq m_0$, the random graph $G_m$ satisfies $\q(G_m)> \q(G) - \eps$ with probability $> 1-\eps$.
\end{corollary}
\begin{corollary} \label{cor.moddiff}
Given $\eps>0$ there exists $c=c(\eps)$ such that the following holds.  
For all $m \geq c n$, the random graph $G_{m}$ satisfies the conclusions of Theorem~\ref{thm.moddiff} when $G_p$ is replaced by~$G_{m}$.
\end{corollary}

These corollaries follow easily from Theorems~\ref{thm.obsmod} and~\ref{thm.moddiff} respectively using Lemma~\ref{lem.GpGm}, which shows that $\q(G_p)$ and $\q(G_m)$ (and similarly for $q_\cA$) are whp close in value for $p=m/e(G)$. The lemma will be proved in Section~\ref{subsec.closeness} (the lemma also shows that these quantities are close in expectation, which will be used in Section~\ref{sec.expmod}).

\begin{figure}
\begin{picture}(170,450) 
	\put(20,250){ 
		\includegraphics[width=0.95\textwidth]{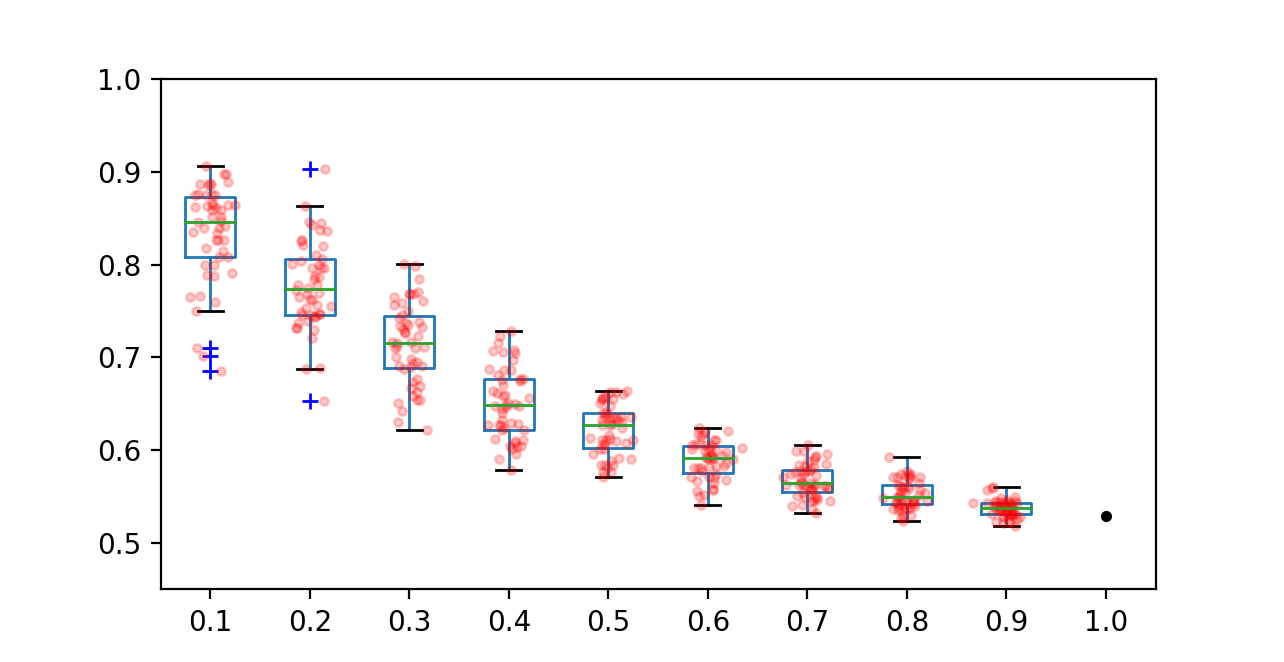}
		\put(-435,95){\rotatebox{90}{$\tilde{q}(G_p)$}} 
	    \put(-313,207){Estimated modularity of sampled graph} 
	    \put(-210,-3){$p$}
		}
	\put(20,2){
		\includegraphics[width=0.95\textwidth]{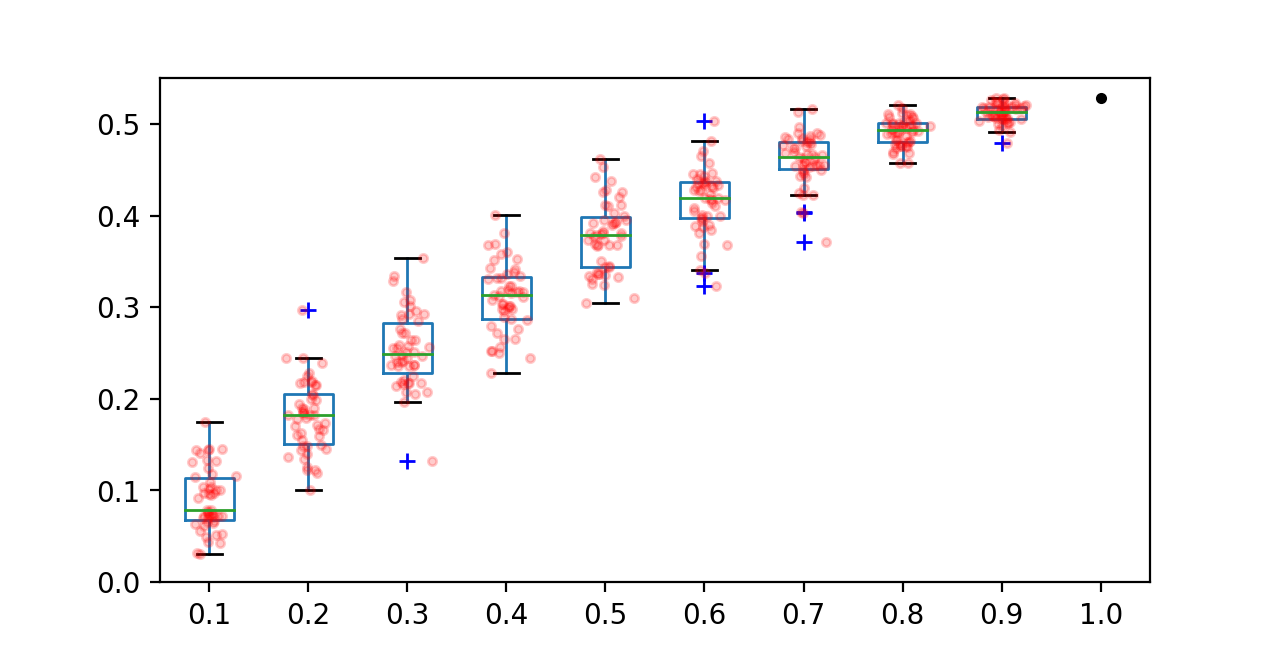}
		\put(-435,95){\rotatebox{90}{$q_{\cA'(G_p)}(G)$}} 
	    \put(-313,223){Modularity score of underlying graph using}
	    \put(-313,207){partition estimated from sampled graph}
	    \put(-210,-3){$p$}
		}
\end{picture}
\caption{ 
Simulation results. The dolphin social network~\cite{lusseau2003emergent} with 62 vertices and 159 edges was taken to be the underlying graph $G$. It is known that $\q(G) = 0.529$ to three decimal places~\cite{nphard}. In the upper part of the figure each red point corresponds to the estimated modularity $\tilde{q}(G_p)$ of an instance of the sampled graph $G_p$.
For each edge probability $p=0.1, 0.2, \ldots, 0.9$, the graph $G_p$ was sampled 50 times. For each sampled graph $G_p$ we took the maximum modularity score of the partitions output by 200 runs of both the Louvain~\cite{louvain} and Leiden~\cite{traag2019leiden} algorithms. The noise in the $x$-axis is to allow one to see the points.\\
In the lower part of the figure we examine, for each random instance of $G_p$, how well the modularity maximising partition of $G_p$ performs as a partition on the underlying graph $G$. For each sampled graph $G_p$ we plot the score $q_{{\cA}'(G_p)}(G)$, where $\cA(G_p)$ is the highest scoring partition on $G_p$ found in 200 runs of Louvain and Leiden and~$\cA'(G_p)$ is the partition modified as in Theorem~\ref{thm.moddiff}(b) (with $\eta=0.05$ in Lemma~\ref{lem.nosmall2}).
See also Figure~\ref{fig.simsBIG} for simulations run on a larger underlying~graph.
}\label{fig.sims}
\end{figure}

\needspace{3\baselineskip}
\subsection{Estimating modularity by vertex sampling (`parameter estimation')}
\label{subsec.estim}
Another commonly considered model of partially observing an underlying graph $G$ is to sample a vertex subset $U$ of constant size $k$ and observe the corresponding induced subgraph~$G[U]$. In Section~\ref{sec.estim} and Theorem~\ref{thm.modest} we consider this model. We find that for dense graphs $G$, with probability at least $1-\eps$ we may estimate the modularity to within an $\eps$ error (that is $|\q(G)-\q(G[U])|< \eps$) by sampling a constant $k=k(\eps)$ number of vertices ; but for graphs with a subquadratic number of edges this result does not hold, modularity is not estimable. 

\begin{theorem} \label{thm.modest}
(a) For fixed $\rho$ with $0<\rho<1$, modularity is estimable for graphs with density at least~$\rho$.
(b) For any given function $\rho(n) = o(1)$, modularity is not estimable for $n$-vertex graphs with density at least $\rho(n)$.
\end{theorem}

\needspace{3\baselineskip}
\subsection{Outline of the paper}
\label{subsec.outline}
The plan of the rest of the paper is as follows. In Section~\ref{sec.further_and_relation}, we first show an application of our results to the stochastic block model in Section~\ref{subsec.bootstrap}, then provide an overview of the further results of this paper in Section~\ref{subsec.furtherresults} and lastly give some background and the relation of this paper to previous results in Section~\ref{subsec.relation}. 

Sections~\ref{sec.fat} to~\ref{sec.proofthm2} give the proofs of the main results of the paper: Section~\ref{sec.fat} gives a crucial preliminary lemma for the proofs, the `fattening lemma'; and Theorem~\ref{thm.obsmod} and Theorem~\ref{thm.moddiff} are proven in Section~\ref{sec.proofthm1} and Section~\ref{sec.proofthm2} respectively. 
(Indeed we also prove versions of these results which take into account the number of parts in the partition.) 

In Section~\ref{sec.robust_close_conc} we give a `robustness' lemma showing that $\q(G)$ and $q_\cA(G)$ do not change much if we change a small proportion of edges; then we prove Lemma~\ref{lem.GpGm} on $G_p$ and $G_m$, which completes the proof of Corollaries~\ref{cor.obsmod} and~\ref{cor.moddiff}; and finally we give related results on concentration of modularity. We prove Theorem~\ref{thm.modest} on estimating modularity by vertex sampling in Section~\ref{sec.estim}. 

The later sections, Section~\ref{sec.undersamp} and~\ref{sec.expmod} contain further related results. In Section~\ref{sec.undersamp} we see that under-sampling tends to lead to over-estimation of modularity (using Theorem~\ref{thm.obsmod}). In Section~\ref{sec.expmod} we show that Theorem~\ref{thm.obsmod} implies results on the expected modularity of random graphs $G_p$ and~$G_m$ with constant average degree. 
In Section~\ref{sec.wmod_new} we give results analogous to Theorems~\ref{thm.obsmod} and \ref{thm.moddiff} for weighted networks and show an application of these results for the stochastic block model. Finally Section~\ref{sec.concl} contains a few concluding remarks and open questions.

In the appendix we give simulations run on a larger underlying graph in Section~\ref{sec.simsBIG}.

\needspace{6\baselineskip}
\section{Further results and background on existing results}
\label{sec.further_and_relation}

\needspace{3\baselineskip}
\subsection{Bootstrapping to sparser graphs and application to stochastic block model.}\label{subsec.bootstrap}
\paragraph*{Modularity preserves signal a little below the connectivity threshold.} 
This property sets modularity apart from  other measures of community in a network. If $s(H)$ is min-cut, normalised min-cut or the spectral gap of the Laplacian of $H$ then $s(H)=0$ if $H$ is disconnected. Thus, given a connected graph $G$, for $s(G_p)$ to approximate $s(G)$ for these measures $s$ we must take $p$ large enough that $G_p$ is likely connected.

To see that our results can hold below the connectivity threshold, consider $G=K_n$ (so $\q(G)=0$) and $p=c/n$ with $c$ a large constant. Then $|\q(G_p)-\q(G)|<\eps$ with probability~$1-\eps$, see for example Theorem 1.3 of~\cite{ERmod}. 
But whp $G_p$ is disconnected -- indeed it will have a linear proportion of the vertices and edges inside the giant component, and a linear proportion outside. Furthermore we may take our underlying graph to be disconnected - for example if $G$ consists of two disjoint equal sized cliques then it has modularity $1/2$, and this will again be approximated by $\q(G_p)$ for $p=c/n$ with some large constant $c$.

\needspace{3\baselineskip}
\paragraph*{Stochastic block model, a little below the connectivity threshold $\Theta(n^{-1} \log n)$.}
Let $k \geq 2$ and $0 \leq q \leq p \leq 1$. There are two versions of the balanced $k$-community stochastic block model, where edges appears independently with probability~$p$ within blocks and probability $q$ between blocks. In the first version, $G_{n, k,p,q}$, the vertex set $V=[n]$ is partitioned deterministically into $k$ sets (blocks) $V_i$  of size $\lfloor n/k \rfloor$ or $\lceil n/k \rceil$; and in the second, $G'_{n, k,p,q}$, each vertex independently and uniformly picks a block to join (so the blocks have size about $n/k$ whp).

We first give an example where bounds at some (not too small) edge density are known from spectral results, and a version of Theorem~\ref{thm.moddiff} allows us to bootstrap these results to sparser models.  After that we present Theorem~\ref{thm.SBMk} concerning the general $k$-community model.

The bootstrapping example involves $G_{n,2,p,q}$. Write $\qtwo$ for the maximum modularity value over all partitions into at most two parts, and note that $\qtwo$ is at least the modularity score of the planted bipartition. Thus by direct calculation - see for example similar calculations in~\cite{koshelev2023modularity, ERmod} - if $n^2p \rightarrow \infty$  then 
\begin{equation}\label{eq.SBM_ub}
\qtwo(G_{n, 2, p, q})\geq \frac{p-q}{2(p+q)}+o(1) \;\;\; \mbox{whp}.
\end{equation}
Using existing spectral results from~\cite{deng2021strong,modexpansion} we may deduce that the lower bound in~\eqref{eq.SBM_ub} is tight for some values of $p,q = \Theta(n^{-1} \log n)$. (We suppress the details here, see Remark~\ref{rem.bootstrap}.) Interestingly, one may then use Proposition~\ref{prop.moddiff_k} (a version of Theorem~\ref{thm.moddiff} for $k$-partitions) to bootstrap these results to sparser graphs, where $p,q =\omega (n^{-1})$ - see Remark~\ref{rem.bootstrap}. Note that this includes values of $p$ and $q$ for which the spectral results do not hold (since for part of the range~$G_{n,2,p,q}$ is disconnected~whp).

The tightness of~\eqref{eq.SBM_ub} tells us that whp the planted partition has asymptotically maximal modularity value over all bipartitions. We now consider the modularity value for the $k$-block model, and see that the planted partition is asymptotically optimal over \emph{all} partitions. 

\begin{theorem} \label{thm.SBMk}
Let $k \geq 2$ be an integer, and let $p=p(n)$ and $q=q(n)$ satisfy $0 \leq q \leq p \leq 1$ and $np \to \infty$ as $n \to \infty$. Then for $G=G_{n, k, p, q}$ or $G=G'_{n, k, p, q}$  
\[ \q(G) = \frac{(p-q) \, (1 - 1/k)}{ p + (k-1)q} +o(1) \;\;\; \mbox{whp},\]
and whp the planted partition has this modularity value.
\end{theorem}

We shall prove Theorem~\ref{thm.SBMk} in Section~\ref{sec.wmod_new}, using a deterministic lemma, Lemma~\ref{lem.weightedSBMk}, together with Theorem~\ref{thm.moddiff_wb}, which is a version of Theorem~\ref{thm.moddiff} with weighted underlying graph.  We note that \cite{bickel2009nonparametric, bickel2015correction} showed that whp the modularity optimal partition will agree with the planted partition except for $o(n)$ vertices (for $p=\omega(1/n)$ and $q=\rho p$ for some fixed $\rho$, $0<\rho <1$).
Thus we could also have proven Theorem~\ref{thm.SBMk} for such $p,q$ 
via robustness results, see Lemma~\ref{lem.robustness}, and the likely modularity value of the planted partition. 

The  recent paper~\cite{koshelev2023modularity} gives explicit bounds on the modularity.
We provide a simple stand-alone proof of Theorem~\ref{thm.SBMk}, following~\cite{koshelev2023modularity} in using weighted graphs. Our result extends that in~\cite{koshelev2023modularity} as we reduce the upper bound there by a factor $1-1/k$ to match the lower bound, and we extend the range of $p$ from $\omega(n^{-1/2})$ to $\omega(n^{-1})$.  Theorem~\ref{thm.SBMk} did not appear in the earlier arXiv version~\cite{samplingold} of our paper.

\needspace{3\baselineskip}
\subsection{Further results in this paper}\label{subsec.furtherresults}

\needspace{1\baselineskip}
\paragraph*{Robustness, concentration and closeness of modularity.}
For either $q_\cA(G)$ or $\q(G)$ modifying a single edge can change the value by at most $2/e(G)$. This was known for $\q(G)$ (see \S5 of \cite{ERmod}). In Section~\ref{subsec.robustness} we prove the $2/e(G)$ bound holds also for any given partition $\cA$, and in Example~\ref{ex.robustness} show examples such that the factor 2 is necessary in both cases. These robustness results give the concentration theorem below - see Section~\ref{sec.robust_close_conc}. 

\begin{theorem}\label{thm.bconc}
There is a constant $\eta>0$ such that for each graph $G$, each partition $\cA=\cA(G)$ and each $0<p<1$ the following holds with $\mu=\mu(G,p) = e(G) p$.  
For each $t \geq 0$ 
\begin{equation*} \pr\Big( \big| \, q_\cA(\Gp) - \E[q_\cA(\Gp)] \, \big| \geq t \Big) < 2 \, e^{-\eta \mu t^2} 
\;\mbox{ and }\;
\pr\Big( \big| \, \q(\Gp) - \E[\q(\Gp)] \, \big| \geq t \Big) < 2 \, e^{-\eta \mu t^2}.\end{equation*}
\end{theorem}
By Theorem~\ref{thm.bconc} we have concentration for $\q(G_p)$ around $\E[\q(G_p)]$ as soon as $e(G)p$ is large.  In more detail,  $\pr( |\q(G_p) - \E[\q(G_p)]| \geq \eps) < \delta$ if $e(G)p > c \eps^{-2} \log \delta^{-1}$ for a (large) constant $c$. But, results on binomial random graphs (Theorem 1.1 (b) of \cite{ERmod}) show that, when $G$ is $K_n$, we need large average expected degree (i.e.\ $e(G)p/v(G)$ large) to avoid $\q(G_p)$ being much larger than $\q(G)$ whp. Thus we get concentration of $\q(G_p)$ as soon as $e(G)p$ is a large constant, however, this may not be concentration  around $\q(G)$ until the average expected degree $e(G)p/v(G)$ is a large constant.

\paragraph*{Under-sampling and over-estimating modularity.} 
In Ecological networks each interaction observed reveals that an edge is present in the underlying network, and the effect of sampling effort can be modelled by taking the observed network after varying numbers of observations. It was noted in multiple papers on ecological networks that a lower sampling effort, under-sampling, can lead to overestimating the modularity of the underlying ecological network
~\cite{vizentin2016influences}. Our paper provides some theoretical explanations - see~Section~\ref{sec.undersamp} for a statement of the result.

\needspace{3\baselineskip}
\paragraph*{Expected modularity of the binomial random graph $G_{n,c/n}$.}
Given a constant $c>0$, for each $n \geq c\,$ we let $\bar{q}(n,c) =  \E[q^*(G_{n,c/n})]$, where $G_{n,c/n}$ is the binomial (or Erd\H{o}s-R\'enyi) random graph with edge-probability $c/n$. It was conjectured in \cite{ERmod} that for each $c>0$, $\bar{q}(n,c)$ tends to a limit $\bar{q}(c)$ as $n \to \infty$, and it was noted there that in this case the function $\bar{q}(c)$ would be continuous in~$c$.   From Theorem~\ref{thm.obsmod} we may deduce that also $\bar{q}(c)$ would be non-increasing in~$c$ - see Section~\ref{sec.expmod}.

\paragraph*{Modularity and edge-sampling on weighted networks.}
Network data which is of interest to cluster often has weights associated with each edge. Though we have stated the modularity score of a partition for binary edge weights it is simple to take the weight of edges inside each part (instead of the number of edges) and to take the degree of a vertex $v$ to be the sum of the weights of the edges incident to $v$ (instead of the number of edges). This weighted modularity is often used, and indeed the popular community detection algorithm Louvain can take weighted networks as input~\cite{louvain}. Our Theorems~\ref{thm.obsmod} and~\ref{thm.moddiff} have analogs for weighted networks -  see Section~\ref{sec.wmod_new}.

\needspace{3\baselineskip}
\subsection{Background on existing results, and our contribution }\label{subsec.relation}
\paragraph*{Modularity : use in community detection.}
Modularity was introduced in~\cite{NewmanGirvan}, and gives a score to each vertex partition (i.e.\ commmunity division) and partitions with higher scores are considered to better capture the communities in the network. It is NP-hard to find a partition~$\cA$ with the highest modularity score (i.e.\ such that $q_\cA(G)=\q(G)$) \cite{nphard} and community detection algorithms do not do this. However, it is fast to compute the modularity of a particular partition and hence it can be feasible to choose which modifications to make to a partition by picking the candidate partition with the highest modularity score. 
Louvain~\cite{louvain} and Leiden~\cite{traag2019leiden} are examples of this. The algorithms are fast and have had success in recovering ground truth communities on real world networks. However, there no theoretical guarantees for either that the partition found is near optimal, though recently~\cite{cohen2020power} showed that a Louvain-like algorithm recovers the communities in the stochastic block model for a wide parameter range. 

Modularity-based clustering algorithms are the most commonly used to detect communities on large network data~\cite{popular, lambiotte2021modularity} - see~\cite{fortunato2010community} and~\cite{porter2009communities} for surveys. 
The widespread use in applications makes modularity an important graph parameter to understand theoretically.

\needspace{3\baselineskip}
\paragraph*{
Informing applied network theory : privacy.}
Sharing network data can lead to privacy concerns - one approach is to share a subsampled graph $G_p$ instead of~$G$~\cite{romanini2021privacy}. A claimed advantage is that $G_p$ retains many properties of the underlying graph $G$ and that parameters of $G$ may be estimated knowing only $p$ and $G_p$. Examples considered in~\cite{romanini2021privacy} include vertex degrees and number of triangles.

Our contribution is to determine when the modularity of $G$ is well approximated by the modularity of~$G_p$. Furthermore, in part (b) of Theorem~\ref{thm.moddiff}, we see that for $p$ large enough we can likely obtain a near optimal partition of the underlying graph $G$ whilst seeing only the shared graph $G_p$. (Here we mean possible in an information theoretic sense - we do not consider the complexity of finding the partition.) 
Given the shared graph $G_p$ and a near optimal partition $\cA$ of the shared graph, we construct a partition which is likely to be near optimal for the underlying (non-shared) graph $G$.

\needspace{3\baselineskip}
\paragraph*{Robustness and percolation : existing results for random subgraphs of fixed graphs.}
Given a graph $G$ with property $\mathcal{P}$ we say that $G$ \emph{robustly} has property~$\mathcal{P}$ if whp $G_p$ has property~$\mathcal{P}$~\cite{sudakov2017robustness}. A seminal such result is in~\cite{krivelevich2014robust} which showed there is an absolute constant $c$ such that for $p \geq c (\log n)/n$ and underlying graph $G$ with minimum degree $n/2$ then whp the observed graph~$G_p$ is Hamiltonian. This was later strengthened to a hitting time result, i.e. taking edges of the underlying graph in a random order,  in~\cite{johansson2020hamilton}. 

Another property that has been considered is expansion - a measure of the number of edges leaving each vertex set relative to the size of that set. For $d$-regular $G$ with good expansion properties the expansion in the giant component of $G_p$ was studied in~\cite{ofek2007expansion} and~\cite{diskin2022expansion}, see also~\cite{diskin2023isoperimetric}. Additionally non-planarity of $G_p$ for $G$ with a specified minimum degree was studied in~\cite{frieze2013non} and for planar underlying graphs \cite{liu2023phase} determines the threshold for $p$ such that $G_p$ is 3-colourable for every planar~$G$. 

Our contribution is to consider robustness of modularity : for $G$ with at least a large constant number of edges and $p$ large enough we get likely lower bounds on $\q(G_p)$ and for $G$ with average degree at least a large constant we also get likely upper bounds on $\q(G_p)$.

\needspace{3\baselineskip}
\paragraph*{Parameter estimation, vertex sampling and network measures : existing work.} 
There is a well developed field of parameter estimation which asks for which parameters is it possible to estimate the parameter of an underlying graph given access to an induced subgraph on a random vertex subset of constant size. For an excellent introduction, including the relation to the theory of graph limits see~\cite{BCLSVjournal, lovasz2012large}. 
There are related results in~\cite {bolla2012testability} :
this paper analysed testability of graph parameters relating to minimising the number of edges between parts, some of which were normalised with respect to the size of the parts but in a different way to modularity, and found some of these parameters not to be testable. Our results complement these by showing that modularity is not testable in general but is testable for dense graphs.

\needspace{3\baselineskip}
\paragraph*{Modularity : existing results for random graphs.}
Recall that the modularity value of a graph is always in the interval $[0,1)$ with higher values taken to indicate a higher extent of community structure. The modularity of the binomial random graph $\Gnp$ exhibits three phases, see~\cite{ERmod}: for sparse graphs (where $np\leq 1+o(1)$) modularity is near~1 whp, for dense graphs (where $np\rightarrow \infty$) it is near~$0$ whp, and inbetween (where $c_1 \leq np \leq c_2$ for constants $1<c_1<c_2$) it is bounded away from $0$ and $1$ whp. Note that for $G=K_n$, we have $np \sim \frac12 e(G)p/n$ and $\q(K_n)=0$~\cite{nphard}, and hence this result on $\Gnp$
in the dense regime is recovered by our Theorem~\ref{thm.moddiff}.
It is also shown in~\cite{ERmod} that, for $1/n < p <0.99$ whp $\q(G_{n,p})=\Theta(1/\sqrt{np})$.

Random regular graphs have received recent attention in~\cite{lichev2022modularity}; and in particular it is shown that  the random cubic graph $G_{n,3}$ whp has modularity in the interval $[0.667, 0.790]$, improving on earlier results. For large $r$ whp $\q(G_{n,r})=\Theta(1/\sqrt{r})$~\cite{treelike, pralat} - note that this is the same order as for a binomial random graph with expected degree $r$.
The random hyperbolic graph whp has modularity asymptotically~$1$ \cite{chellig2022modularity}, and so does the preferential attachment tree, though if $h\geq 2$ edges are added at each step as in the model $G_{n}^h$ then $\Omega(1/\sqrt{h}) = \q(G_n^h) <0.94$ whp~\cite{pralat}. The modularity of the stochastic block model, and of a degree-corrected version, have been considered in~\cite{koshelev2023modularity, bickel2009nonparametric, bickel2015correction, zhao2012consistency}, see also Theorem~\ref{thm.SBMk}.


\needspace{6\baselineskip}
\section{The fattening lemma for vertex partitions}
\label{sec.fat}

Given a graph $G$ and $0<\eta \leq 1$, we say that a vertex partition $\cA$ is $\eta$-\emph{fat} (or $\eta$-\emph{good}) if each part has volume at least $\eta \, \vol(G)$. We shall describe a greedy algorithm which, given a graph~$G$, $0<\eta \leq 1$ and a vertex partition $\cB$, amalgamates some parts in $\cB$ to yield an $\eta$-fat partition $\cA=\cA(G,\eta,\cB)$ with modularity score at most a little less than that of $\cB$. Note that $|\cA| \leq |\cB|$.
\begin{lemma}[The fattening lemma]\label{lem.nosmall2}
For each non-empty graph $G$ and each $0<\eta \leq 1$, there is an $\eta$-fat partition $\cA$ of $V(G)$ such that $q_{\cA}(G) > \q(G) - 2\eta$. Indeed, given any partition $\cB$ of $V(G)$, the greedy amalgamating algorithm uses a linear number of operations and constructs an $\eta$-fat partition $\cA =\cA(G,\eta,\cB)$ such that $q_{\cA}(G) > q_{\cB}(G) - 2\eta$.
\end{lemma}

For comparison, note the neat result of Dinh and Thai~\cite{dinh2011finding}  
that, for each graph $G$ and positive integer $k$, we have
\begin{equation}\label{eq.DinhThai} \qk(G) \geq \q(G)\,(1-1/k),\end{equation}
where  $\qk(G)$ is the maximum modularity score over partitions with at most $k$ parts. Observe that if $\cA$ is an $\eta$-fat partition for $G$ and $k = \lfloor 1/\eta \rfloor$, then $q_{\cA}(G) \leq q_{\leq k}(G)$. However note neither approximation result implies the other.
The constant 2 in Lemma~\ref{lem.nosmall2} is best possible, as shown by the following example.

\begin{example}\label{ex.2bestpossible}
Fix $0<b<2$.  Let the odd integer $k$ be sufficiently large that
\[ 1- \tfrac1{k} > b\, (\tfrac12 - \tfrac1{6k}),\]
so there exists $\eta$ such that $\eta  >\tfrac12 - \tfrac1{6k}$ and
$1-\tfrac1{k} > b \eta$.
Let the graph $G$ consist of $k$ disjoint triangles. Thus $G$ has $3k$ vertices and $\vol(G)=6k$.
Since $\eta>\tfrac{1}{2} - \tfrac1{6k}$ the only $\eta$-fat partition for $G$ is the trivial partition $\TT$, with modularity score 0.  Also $\q(G)=1- \tfrac1{k}$ (achieved with the connected components partition), see Proposition~1.3 of \cite{treelike}.
Thus
\[ \q(G) -q_{\TT}(G) = 1- \tfrac{1}{k} > b \eta,\]
so in Lemma~\ref{lem.nosmall2} we cannot replace 2 by $b$.
\end{example}

To prove Lemma~\ref{lem.nosmall2}, it of course suffices to prove the second part.  The greedy amalgamating algorithm to yield a good $\eta$-fat partition is essentially a greedy algorithm for numbers, and we describe that first. The main step in the number greedy algorithm involves bipartitions, and the following well known problem.
Before stating the problem, let us note some standard easy inequalities which are useful for the number (bi-) partitioning problem and for considering degree tax.
Given $x_1,\ldots,x_k \geq 0$ with $\sum_i x_i = s$, we have \begin{equation}\label{eqn.obs} s \, \min_i x_i \leq s^2/k \leq \sum_i x_i^2 \leq s\, \max_i x_i.
\end{equation}

\needspace{3\baselineskip}
\paragraph*{The number (bi-) partitioning problem} 
Given a positive integer $n$ and a positive $n$-vector $\xvec=(x_1,\ldots,x_n)$ with $\sum_i x_i =1$, let
\[ \lambda(\xvec) = \max_{A \subseteq [n]} \, \min \, \{ \sum_{i \in A} x_i, 1- \sum_{i \in A} x_i \}.\]
Thus $0 \leq \lambda(\xvec) \leq \frac12$.  Determining $\lambda(\xvec)$ is the \emph{number partitioning problem}, one of Garey and Johnson's six basic NP-hard problems. It is well known that if each $x_i \leq \frac13$ then $\lambda(\xvec) \geq \frac13$.  (To see this, observe that as we increase $j$ the last partial sum $x_1 + x_2 + \ldots + x_j$ which is at most $\frac23$ is at least $\frac13$.  This result also follows immediately from~(\ref{eqn.gamma1}) below.) 
  
We might expect that $\lambda(\xvec)$ is large when $\sum_ix_i^2$ is small, that is when $1\!-\!\sum_ix_i^2$ is large. We would like to find the largest constant $\alpha$ such that always $\lambda(\xvec) \geq \alpha (1\!-\!\sum_ix_i^2)$.  We must have $\alpha \leq \frac12$: for example if $n>1$ is odd and each $x_i=1/n$, then $\lambda(\xvec)= (n-1)/2n = \frac12(1-1/n) = \frac12(1-\sum_ix_i^2)$.  We shall show that $\frac12$ is the right answer for $\alpha$, that is, we always have $\lambda(\xvec) \geq \frac12(1-\sum_ix_i^2)$; and further there is a simple greedy bi-partitioning algorithm which achieves this.

Assume that the elements are ordered so that $x_1 \geq x_2 \geq \cdots \geq x_n>0$.  We have two bins, and we add the elements one by one to a smaller of the bins.  In detail, the greedy bi-partitioning algorithm is as follows.   Initially set $A=B=\emptyset$.  For $j=1,\ldots,n$, if $\sum_{i \in A} x_i  \leq \sum_{i \in B} x_i $ then insert $j$ into $A$, else insert $j$ into $B$.  At the end, let $\gamma(\xvec)= \min\{\sum_{i \in A} x_i , \sum_{i \in B} x_i\}$. The algorithm clearly uses at most $n$ comparisons and $n$ additions.
\begin{lemma} \label{lem.gamma}
\begin{equation} \label{eqn.gamma}
\gamma(\xvec) \geq \tfrac12 (1-\sum_ix_i^2).
\end{equation}
\end{lemma}
Note by~\eqref{eqn.obs} we have $\sum_ix_i^2 \leq x_1$, so~(\ref{eqn.gamma}) gives 
\begin{equation} \label{eqn.gamma1}
\gamma(\xvec) \geq \tfrac12(1-x_1).
\end{equation}

\begin{proof}[Proof of Lemma~\ref{lem.gamma}]
When the algorithm has finished, let $T$ be the total of the values $x_i$ in the bin containing $x_n$; and let $\bar{T}=1-T$, the total in the other bin. 
We shall use induction on $n$. The result is trivial if $n=1$, since both sides of~(\ref{eqn.gamma}) are 0.  Let $n \geq 2$ and assume that the result holds for all inputs of length $n-1$.  We shall consider two cases. 

\begin{itemize}  
    \item[(a)]  Suppose that $T \geq \frac12$, so $T \geq \bar{T}$.  Then $T - \bar{T} \leq x_n$, so   $\gamma(\xvec) = \bar{T} \geq \frac12 (1-x_n)$. But  $\sum_i x_i^2 \geq x_n$ by~\eqref{eqn.obs}, so~(\ref{eqn.gamma}) holds (without using the induction hypothesis).

    \item[(b)] Suppose that $T < \frac12$, so $\gamma(\xvec) = T < \bar{T}$.
    Let $x= x_n$ and $s = \sum_{i=1}^n x_i^2$,  so $0<x, s <1$.
    Let $\yvec=(y_1,\ldots,y_{n-1})$ where $y_i=x_i/(1-x)$ (and note that $\sum_i y_i=1$).  By the induction hypothesis, 
    \[\gamma(\yvec) \geq \tfrac12 (1-\sum_i y_i^2) = \tfrac12 (1- \tfrac{s-x^2}{(1-x)^2}).\] 
    But $\gamma(\xvec) = x+ (1-x) \gamma(\yvec)$. Hence
\begin{eqnarray*}
	\gamma(\xvec)
	& \geq & 	x + \tfrac12 (1-x)(1- \tfrac{s-x^2}{(1-x)^2}) \;\; = \;\; \tfrac12 (1+x - \tfrac{s-x^2}{1-x})\\
	& = & 		\tfrac12 \tfrac{1-x^2-s+x^2}{1-x} \;\; = \;\; \tfrac12 \tfrac{1-s}{1-x} \;\;> \;\; \tfrac12 (1-s).
    \end{eqnarray*}
Hence, again~(\ref{eqn.gamma}) holds. \end{itemize}
Now we have established the induction step, and thus proven the lemma.\end{proof}
Next we consider partitioning numbers into several parts.
 Let $n \geq 1$ and let $\xvec = (x_1,\ldots,x_n)$ with each $x_i>0$ and $\sum_ix_i=1$.  Let $k \geq 1$ and let $\cA=(A_1,\ldots,A_k)$ be a partition of $[n]$. 
Let $S_j=\sum_{i\in A_j} x_i$ for each $j=1,\ldots,k$.
The corresponding \emph{cost} for $\cA$ and $\xvec$ is
\[ c(\cA,\xvec) = \sum_{j\in [k]} {S_j}^2 - \sum_{i \in [n]} x_i^2. \]
Observe that $0 \leq c(\cA,\xvec) <1$.
Given $0<\eta \leq 1$, we say that the partition $\cA$ is $(\xvec,\eta)$-\emph{fat} if $S_j \geq \eta$ for each~$j$.  Observe that the trivial partition is always $(\xvec,\eta)$-fat.

The number greedy partitioning algorithm starts with the trivial partition.  While there is a part $A$ such that, setting $S = \sum_{i \in A} x_i$ and $\yvec=(x_i/S :i \in A)$ we have $\gamma(\yvec) \geq \eta/S$, it picks the first such part $A$ and uses the greedy number bi-partitioning algorithm to split $A$ into two parts each with sum at least $\eta$.

\begin{lemma} \label{lem.2etax}
Let $n \geq 1$ and let $\xvec = (x_1,\ldots,x_n)$ with each $x_i>0$ and $\sum_ix_i=1$; and let $0<\eta \leq 1$.  Then the number greedy partitioning algorithm finds an $(\xvec,\eta)$-fat partition $\cA$ of~$[n]$ with $c(\cA,\xvec) < 2 \eta$, using $O(n)$ operations.
\end{lemma}
\begin{proof}
Let the final partition obtained by the greedy  amalgamating algorithm be $\cA =(A_1,\ldots,A_k)$ (for some $k \geq 1$).
Fix $j \in [k]$.  Let $\xvec^{(j)}=(x_i/S_j: i \in A_j)$.  Let $c_j$ denote the cost of the trivial partition of $A_j$ with $\xvec^{(j)}$, so $c_j = 1-\sum_{i \in A_j} (x_i/S_j)^2 \geq 0$. 
By Lemma~\ref{lem.gamma},
$\gamma(\xvec^{(j)}) \geq \tfrac12 (1-\sum_{i \in A_j} (x_i/S_j)^2) = \tfrac12 c_j$.  But since the algorithm stopped with $\cA$, we have $\gamma(\xvec^{(j)}) < \eta/S_j$.  Hence $c_j < 2\eta/S_j$ for each $j \in [k]$, and so
\begin{eqnarray*}
   c(\cA,\xvec)
&=&
   \sum_{j\in [k]} ({S_j}^2 -\!\sum_{i \in A_j} x_i^2) \; = \; \sum_{j\in [k]} {S_j}^2(1 -\!\sum_{i \in A_j} (x_i/S_j)^2)\\
&=&
   \sum_{j\in [k]} {S_j}^2 c_j \; < \sum_{j\in [k]} {S_j}^2 (2 \eta/S_j) \; = \; 2 \eta,
\end{eqnarray*}
as required. \end{proof}

Finally, let us deduce the fattening lemma, Lemma~\ref{lem.nosmall2}, from Lemma~\ref{lem.2etax}.  The greedy amalgamating algorithm which we apply to the partition $\cB =(B_1,\ldots,B_n)$ of $V(G)$ is essentially the number greedy partitioning algorithm applied to the partition of $[n]$ into singletons $\{i\}$, where the weight of $i$ is $\vol(B_i)/\vol(G)$.

\begin{proof}[Proof of Lemma~\ref{lem.nosmall2}]
Let $\cB=(B_1,\ldots,B_n)$ be a partition of $V(G)$; and let $\xvec=(x_1,\ldots,x_n)$ where $x_i=\vol(B_i)/\vol(G)$.  By Lemma~\ref{lem.2etax} the number greedy partitioning algorithm finds an $(\xvec,\eta)$-fat partition $\cA=(A_1,\ldots,A_k)$ of $[n]$ with $c(\cA,\xvec) < 2 \eta$.  Now $\cA$ yields the partition $\tilde{\cA}=(\tilde{A}_1,\ldots, \tilde{A}_k)$ of $V(G)$, where $\tilde{A}_j= \cup_{i \in A_j} B_i$; and $\vol(\tilde{A}_j) = (\sum_{i \in A_j} x_i)\, \vol(G)$, so $\tilde{\cA}$ is $\eta$-fat for $G$.
Also $q_{\tilde{\cA}}^E(G) \geq q_{\cB}^E(G)$, and
$q_{\tilde{\cA}}^D(G) - q_{\cB}^D(G) = c(\cA,\xvec)$. Thus
\begin{eqnarray*}
 q_{\tilde{\cA}}(G) & \geq &
 q_{\cB}^E(G) - q_{\tilde{\cA}}^D(G) \; = \; q_{\cB}^E(G) - (c(\cA,\xvec) + q_{\cB}^D(G))\\
& = &
q_{\cB}(G) - c(\cA,\xvec) \; > \; q_{\cB}(G) - 2 \eta.
\end{eqnarray*}
This completes the proof of  Lemma~\ref{lem.nosmall2}.
\end{proof}


\needspace{6\baselineskip}
\section{Modularity of $G_p$ with large expected number of edges}\label{sec.proofthm1}
\subsection{Proof of Theorem~\ref{thm.obsmod}}\label{subsec.proofthm1}
We shall use a tail bound for random variables with a binomial or similar distribution.  We use a variant of the Chernoff bounds, which follows for example from Theorems~2.1 and 2.8 of \cite{JLRbook} by considering $S/b$. 
In this and the next section we shall always have $b=1$ or $b=2\,$: we will need general $b$ in Section~\ref{sec.wmod_new}.
\begin{lemma} \label{lem.conc}
Let $b>0$, and let random variables $X_1, \ldots,X_k$ be independent, with $0 \leq  X_j \leq b$ for each $j$. Let $S = \sum_{j=1}^{k} X_j$ and $\mu = \E(S)$. Then 
\[ 
\pr( |S - \mu| \geq \eps \mu) \leq 2 e^{- \tfrac13  \eps^2 \mu/b} \;\;\; \mbox{ for } 0 < \eps \leq 1; \]
or equivalently
\[ \pr(|S-\mu| \geq x \sqrt{\mu}) \leq 2 e^{-\frac13 x^2/b}  \;\;\; \mbox{ for } 0 < x \leq \sqrt{\mu}\,.
\]
\end{lemma}

\begin{proof}[Proof of Theorem~\ref{thm.obsmod}]  First note that we may assume that $0<\eps<1$ and $\q(G)\geq \eps$. Let $\eta= \eps/4$. 
By Lemma~\ref{lem.nosmall2} there is an $\eta$-fat partition $\cA=(A_1,\ldots,A_k)$ for $G$ such that $q_{\cA}(G) \geq \q(G)- 2 \eta$ (thus $q_{\cA}(G)\geq 2\eta$ by our choice of $\eta=\eps/4$).
Observe that the number $k$ of parts in $\cA$ is at most $\eta^{-1}$, and thus $q_{\cA}^D(G) \geq \eta$ by~(\ref{eqn.obs}).
To prove Theorem~\ref{thm.obsmod} we consider first the edge contribution and then the degree tax. Let $t = (e(G)p)^{1/2}$. By Lemma~\ref{lem.conc} (with $b=1$), for $0 \leq x \leq t$ 
\begin{equation} \label{eqn.edge}
\pr( |e(G_p) - e(G)p |  \geq e(G) p \cdot x/t) \leq 2 e^{-x^2 /3}\,.
\end{equation}
For a graph $H$ on $V(G)$, let $\eint(H)$ denote the number of `internal' edges of $H$ within the parts of~$\cA$, so that $q_{\cA}^E(H)=\eint(H)/e(H)$.   Then by Lemma~\ref{lem.conc} as above, for $0 < x \leq (\eint(G)p)^{1/2}$
\[ \pr( \eint(G_p) \leq \eint(G)p\, (1 - x \,({\eint(G)p})^{-1/2}) \leq 2 e^{-x^2/3}\,.\]
But
\[ q_{\cA}^E(G) > \q(G)- 2 \eta + q_{\cA}^D(G) \geq \q(G) - \eta \geq 3 \eta,\]
so $\eint(G) \geq 3\eta\, e(G)$. Hence, for $0<x \leq (3\eta)^{1/2} t$,
with probability at least $1- 2e^{-x^2/3}$
\begin{equation} \label{eqn.int}
\eint(G_p) \geq \eint(G) \, p \, (1-x ( 3 \eta\, e(G)p)^{-1/2}) = \eint(G) \, p \, (1- (3 \eta)^{-1/2}x/t).
\end{equation}
Also, by~(\ref{eqn.edge}), for $0<x \leq t$ we have $e(G_p) \leq e(G)p \, (1 + x/t)$ with probability at least $1-2e^{-x^2/3}$.
Hence by~\eqref{eqn.int}, for $0<x \leq (3\eta)^{1/2} t$, with probability at least $1- 4e^{-x^2/3}$
\begin{equation} \label{eqn.qE}
  q_{\cA}^E(G_p) = \frac{\eint(G_p)}{e(G_p)} \geq q_{\cA}^E(G) \frac{1-(3 \eta)^{-1/2}x/t}{1+x/t}.
\end{equation}
The degree tax is only a little more complicated. 
Let $X_i= \vol_{G_p}(A_i)$, so $\E[X_i]= \vol_G(A_i)p$.
By Lemma~\ref{lem.conc} with $b=2$, for $0 < x \leq (\vol_G(A_i)p)^{1/2}$, with probability at least $1- 2e^{- x^2 /6}$ we have
\[ X_i \leq \vol_G(A_i)p \, \big(1+ x \,\big({\vol_G(A_i)p}\big)^{-1/2}\big)
\,.\]
But, since $\vol_G(A_i) \geq \eta \, \vol(G)$,
\[ \big(\vol_G(A_i)p\big)^{-1/2} \leq \big(\eta\, \vol(G)p \big)^{-1/2} = (2\eta)^{-1/2}/t;\]
and so, for $0 < x \leq (2\eta)^{1/2} t$, with probability at least $1-2e^{-x^2 /6}$ we have
\[ X_i \leq \vol_G(A_i)p \, \big(1+ (2 \eta)^{-1/2} x/t \big)\,. \]
Recall that $\cA$ has $k$ parts and thus for $0 < x \leq (2\eta)^{1/2} t$, with probability at least $1- 2k e^{-x^2 /6}$
\[ \sum_i X_i^2 \leq \sum_i \vol_G(A_i)^2 p^2 \, \big( 1+  (2 \eta)^{-1/2} x/t \big)^2.\]
Also, by~(\ref{eqn.edge}), for $0<x \leq t$, with probability at least $1- 2e^{-x^2/3}$ we have $e(G_p) \geq e(G)p (1 - x/t)$ and so $\vol(G_p)^2 \geq \vol(G)^2 p^2 (1 - x/t)^2$. Hence, for $0 < x \leq (2\eta)^{1/2} t$, with probability at least $1- 2(k+1) e^{-x^2 /6}$
\begin{equation} \label{eqn.qD}
 q_{\cA}^D(G_p) \leq  q_{\cA}^D(G) \left( \frac{1+ (2 \eta)^{-1/2} x/t}  {1 - x/t} \right)^{2}.
 \end{equation}
Now, for $0 < x \leq (2\eta)^{1/2} t$, with probability at least $1- 2(k+3) e^{-x^2 /6}$, both~(\ref{eqn.qE}) and~(\ref{eqn.qD}) hold. (With a little more care we could replace the factor $(k+3)$ here by $(k+2)$ but that would not make a significant difference.)
Since $k \leq\eta^{-1} = 4/\eps$, the failure probability here is at most $2 (4/\eps +3) e^{-x^2 /6}$.  
Choose $x=x(\eps)$ sufficiently large that this probability is at most~$\eps$; and note that we may take $x= \Theta(\sqrt{\log \eps^{-1}})$.  In the next paragraph we shall choose $t_0 \geq (2 \eta)^{-1/2} x$, so~(\ref{eqn.qD}) holds for our choice of $x$ and any $t \geq t_0$.

Consider the ratios on the right sides of the inequalities~(\ref{eqn.qE}) and~(\ref{eqn.qD}).  Choose $t_0$ sufficiently large that for $t \geq t_0$
\begin{equation} \label{eqn.geq}
 \frac{1-(3 \eta)^{-1/2}x/t}{1+x/t} \geq 1-\eta
\end{equation}
and
\begin{equation} \label{eqn.leq}
  \left( \frac{1+ (2 \eta)^{-1/2} x/t}  {1 - x/t} \right)^{2} \leq 1+\eta.
\end{equation}
Note that the left side in~(\ref{eqn.geq}) is increasing in $t$, and the left side in~(\ref{eqn.leq}) is decreasing in $t$, so we just need the inequalities to hold for $t_0$. Rearranging, we see the inequality~\eqref{eqn.geq} is equivalent to
\[ t\geq x(3^{-1/2}\eta^{-3/2}+\eta^{-1}-1) .\]
and thus we must have $t_0\geq 3^{-1/2}\eta^{-3/2}x$.
In fact we may take $t_0=\Theta( \eta^{-3/2} x)=\Theta(\eps^{-3/2} \sqrt{\log \eps^{-1}} )$ and this will ensure both~\eqref{eqn.geq} and~\eqref{eqn.leq} hold.

Finally, set $c=t_0^2$, so $c= \Theta(\eps^{-3} \log \eps^{-1})$; and let us check that this value for $c$ works.  With probability at least $1-\eps$ both~(\ref{eqn.qE}) and~(\ref{eqn.qD}) hold.
Suppose that both~(\ref{eqn.qE}) and~(\ref{eqn.qD}) do hold: if also $e(G)p \geq c$ then $t \geq t_0$, so~(\ref{eqn.geq}) and~(\ref{eqn.leq}) hold; and then
\begin{equation}\label{eq.thm1prooflastline} q_{\cA}(G_p) \geq q_{\cA}^E(G)(1-\eta) -  q_{\cA}^D(G)(1+\eta) > q_{\cA}(G) - 2 \eta \geq \q(G) - \eps.\end{equation}
Hence, if $e(G)p \geq c$, then $q_{\cA}(G_p) > \q(G)-\eps$  with probability at least $1-\eps$, as required. This completes the proof of Theorem~\ref{thm.obsmod}.
\label{end.proofthm1}\end{proof}
\smallskip
In the above proof of Theorem~\ref{thm.obsmod} we took $c= \Theta(\eps^{-3} \log \eps^{-1})$.
If we used a more detailed and precise form of Lemma~\ref{lem.conc} - see for example Theorem 21.6 of~\cite{frieze2015book} - we could improve several bounds in the proof but we would not improve the asymptotic estimate for $c$. If we simply used Chebyshev's inequality we find we can take $c = \Theta(\eps^{-5})$.

The following example shows that, for the conclusion of Theorem~\ref{thm.obsmod} to hold, $c(\eps)$ must be at least $\Omega(\eps^{-1} \log \eps^{-1})$ as $\eps \to 0$.

\begin{example}\label{ex.obsmodconstant}
Let $\eps>0$ be small.  Let the integer $c = c(\eps)$ be sufficiently large that the conclusion of Theorem~\ref{thm.obsmod} holds.
Let $m = 2c$ and $k= \lceil \eps m \rceil$, and assume that $k+1<m$. Let $G$ be the $m$-edge graph consisting of a star with $m-k$ edges together with $k$ isolated edges.  Then the connected components partition is the optimal partition for~$G$, since each component has modularity~0~\cite{modexpansion}. Thus 
\[ \q(G) = 1- \tfrac{(m-k)^2 +k}{m^2} = \tfrac{2k}{m} - \tfrac{k^2+k}{m^2} > \tfrac{k}{m} \geq \eps. \]
Now let $p=\frac12$, and note that $e(G)p = mp \geq c$.  The probability that $G_p$ contains none of the $k$ isolated edges is $(1-p)^k > 2^{-\eps m-1}$.  But any star has modularity 0, so
$ \pr(\q(G_p)=0) > 2^{-\eps m -1}$.  Thus we must have
$2^{-\eps m} < 2 \eps$, and so $c > \frac12 \eps^{-1} \log_2 (2\eps)^{-1}$.\end{example}

\needspace{3\baselineskip}
\subsection{A $k$-part analogue to Theorem~\ref{thm.obsmod}}
\label{subsec.obsmod_k}

Recall that $\qk(G) = \max_{|\cA| \leq k} q_\cA(G)$, the maximum value of $q_\cA(G)$ over all vertex partitions $\cA$ with at most $k$ parts. We note here a variant of Theorem~\ref{thm.obsmod} where we replace each instance of $\q$ by $\qk$. (Observe that the inequality~\eqref{eq.DinhThai} of Dinh and Thai with large $k$ shows that Proposition~\ref{prop.obsmod_k} implies Theorem~\ref{thm.obsmod}. A similar connection will hold for Proposition~\ref{prop.moddiff_k} and Theorem~\ref{thm.moddiff}.)

\begin{proposition}\label{prop.obsmod_k}
Given $\eps>0$ there exists $c=c(\eps)$ such that the following holds. For each graph~$G$ and probability $p$ such that $e(G)p \geq c$, and for each $k\geq 2$, the random graph $G_p$ satisfies $\qk(G_p) > \qk(G)-\eps$ with probability $\geq 1-\eps$.
\end{proposition}

\begin{proof}
The proof is almost immediate following that of Theorem~\ref{thm.obsmod}. As in that proof, first note we may assume that $0 < \eps < 1$ and $\qk(G)\geq \eps$. Let $\cA$ be a partition with at most $k$ parts and such that $q_{\cA}(G)=\qk(G)$. Let $\eta=\eps/4$ and then by Lemma~\ref{lem.nosmall2} there is an $\eta$-fat partition $\cA'$ such that $q_{\cA'}(G)\geq q_{\cA}(G) - 2\eta $ and such that $\cA$ is a refinement of $\cA'$. Note in particular this means that~$|\cA'|\leq |\cA|\leq k$ and so $\cA'$ has at most $k$ parts. 

Then we may proceed exactly as in the proof of Theorem~\ref{thm.obsmod} and the analogue of line \eqref{eq.thm1prooflastline} would say that then 
with probability at least $1-\eps$ we have
\begin{equation}\notag q_{\cA'}(G_p) > q_\cA(G) - \eps.\end{equation}
But now, since $q_\cA(G)=\qk(G)$ and since $|\cA'|\leq k$ this implies that with probability at least $1-\eps$ we have $\qk(G_p) \geq q_{\cA'}(G_p) > \qk(G)-\eps$ as required.
\end{proof}

\needspace{6\baselineskip}
\section{Modularity of $G_p$ with large expected degree}\label{sec.proofthm2}

In this section we first prove Theorem~\ref{thm.moddiff}, and then we give a $k$-part analogue Proposition~\ref{prop.moddiff_k} and use it to complete the discussion of the bootstrapping example from Section~\ref{subsec.bootstrap}.

\subsection{Proof of Theorem~\ref{thm.moddiff}}\label{subsec.proofofthm2}
The rough idea of the proof of Theorem~\ref{thm.moddiff} is that we can use the fattening lemma, Lemma~\ref{lem.nosmall2}, to bound the probability that a vertex partition behaves badly by the probability that a fat vertex partition behaves similarly badly, and we can use probabilistic methods to handle fat partitions. However, even after the streamlining provided by Lemma~\ref{lem.nosmall2} the proof is quite intricate and we will need to define a large number of events, many of which are `bad events' parameterised by deviation from the ideal case.

\bigskip

\begin{proof}[Proof of Theorem~\ref{thm.moddiff}]
Let $\eps>0$. We shall choose $c=c(\eps)$ sufficiently large that certain inequalities below hold. It will suffice to choose $c > K \eps^{-3} \log \eps^{-1}$ for a sufficiently large constant $K$.
Let $G=(V,E)$ be a (fixed) $n$-vertex graph and let $0<p<1$; and assume that $e(G)p \geq cn$.

We now define the `bad' events $\cB_0$, $\cB_1$, $\cB_2$ and $\cE_0$ (there will be more later!). 
Let $\cB_0$ be the event $\{\q(G_p) < \q(G) -\eps \}$. Let $\eta= \tfrac{\eps}{9}$. Let $\cE_0$ be the event that there is a partition $\cA_0$ such that $q_{\cA'_0}(G) < q_{\cA_0}(G_p) - \eps$, where $\cA'_0=\cA(G_p,\eta,\cA_0)$ as in Lemma~\ref{lem.nosmall2}. 
Let $\cQ$ be the set of partitions which are $\frac{\eta}2$-fat for $G$.
Let $\cB_1$ be the event that for some $A \subseteq V$ with $\vol_{G}(A) < \frac{\eta}2 \vol(G)$ we have $\vol_{G_p}(A) \geq \eta \vol(G_p)$ (that is, $A$ is not $\eta/2$-fat for $G$ but is $\eta$-fat for $G_p$).
Observe that if some partition $\cA \not\in \cQ$ is $\eta$-fat for $G_p$, then $\cB_1$ holds.
Let $\cB_2$ be the event that there is a partition $\cA \in \cQ$ such that $q_{\cA}(G_p) > q_{\cA}(G)+\tfrac{\eps}{2}$.

We shall show that \begin{equation} \label{eqn.a}
 \{ \q(G_p) > \q(G) + \eps \} \subseteq \cB_1 \cup \cB_2
\end{equation}
and
\begin{equation} \label{eqn.b}
 \cE_0 \subseteq \cB_1 \cup \cB_2.
\end{equation}
Note that (\ref{eqn.a}) yields $\{|\q(G_p) - \q(G)| > \eps\} \subseteq \cB_0 \cup \cB_1 \cup \cB_2$. It will follow that the probability that (a) or (b) in the theorem fails is at most  $\pr(\cB_0) + \pr(\cB_1) + \pr(\cB_2)$.
By Theorem \ref{thm.obsmod}, if $c n$ is sufficiently large (depending only on $\eps$) then $\pr(\cB_0) < \eps/3$. We shall show that, if we choose $c$ sufficiently large (depending only on $\eps$) then also 
\begin{equation} \label{eqn.B1}
\pr(\cB_1) < \eps/3 
\end{equation}
and
\begin{equation} \label{eqn.B2}
 \pr(\cB_2) < \eps/3  
\end{equation}
which will complete the proof of Theorem~\ref{thm.moddiff}.  Next come the details: the proofs of \eqref{eqn.a}-\eqref{eqn.B2}.
We focus initially on part (a) and the containment~(\ref{eqn.a}).  

\begin{proof}[Proof of~(\ref{eqn.a})]
By Lemma~\ref{lem.nosmall2} (and since $2 \eta \leq \eps/2$),
\begin{eqnarray*}
\{ \q(G_p) > \q(G)+\eps\}
& \subseteq &
  \{\exists\, \cA\: \eta\!-\!\mbox{fat for} \, G_p : q_{\cA}(G_p) > \q(G)+\tfrac{\eps}{2}\} \\
& \subseteq &
  \{\exists\, \cA\: \eta\!-\!\mbox{fat for} \, G_p : q_{\cA}(G_p) > q_{\cA}(G)+\tfrac{\eps}{2}\} \\
& \subseteq &
\{ \exists\, \cA \not\in \cQ : \cA \mbox{ is } \eta\!-\!\mbox{fat for} \, G_p \} \cup \cB_2\\
& \subseteq & \cB_1 \cup \cB_2.\end{eqnarray*}
Thus the containment~(\ref{eqn.a}) holds.
\end{proof}

We show next that the inequality~(\ref{eqn.B1}) holds.  Let us observe first that if $e(G)p/v(G) \geq c$ for some (large) $c$ then $v(G) \geq 2c\,$; so we are concerned only with large graphs.

\begin{proof}[Proof of~(\ref{eqn.B1})]
Let $A \subseteq V$ have $\vol_{G}(A) <  \frac{\eta}2 \vol(G)$.  Suppose that the edges within $A$ are $e_1,\ldots,e_j$ and the edges with exactly one end in $A$ are $e_{j+1},\ldots,e_k$.
For $i=1,\ldots,j$ let $X_i$ be 2 if $e_i$ is in $G_p$ and 0 otherwise; and for $i=j+1,\ldots,k$ let $X_i$ be 1 if $e_i$ is in $G_p$ and 0 otherwise.  Let $S = \sum_{i=1}^{k} X_i$, and let $\mu = \E(S)$. Then $\vol_{G_p}(A) = S$ and $\mu = \vol_G(A)p  < \frac{\eta}2 \vol(G) p$. Form $S'$ by adding further independent random variables $X_i$ with $0 \leq X_i \leq 1$ so that $\mu' := \E[S'] = \frac{\eta}2 \vol(G) p$.
Now by Lemma~\ref{lem.conc} applied to $S'$ (with $b=2$)
\begin{eqnarray*}
\pr( \vol_{G_p}(A) \geq \tfrac23 \eta \, \vol(G) p) 
& = &
 \pr( S \geq \tfrac23 \eta \, \vol(G) p)\\
& \leq &
\pr( S' \geq \tfrac23 \eta \, \vol(G) p)
\; = \;
\pr(S' \geq \tfrac43 \mu')\\
& \leq & 
2 e^{-\frac1{54} \mu' } \; = \; 2 e^{-\frac{\eta}{108} \vol(G) p } \; \leq \; 2 e^{-\frac{\eta}{54 cn }}.
\end{eqnarray*}

Let $\cB_3$ be the (bad) event that $\vol_{G_p}(A) \geq \tfrac23 \eta\, \vol(G)p\,$ for some $A \subseteq V$ with $\vol_G(A) < \frac12 \eta \vol(G)$.  Then, by a union bound, $\pr(\cB_3) \leq 2^n \cdot  2 e^{-\frac{\eta}{108} cn}$,  which is at most $\tfrac{\eps}6$ if $c \geq K/\eps$ for a sufficiently large constant $K$. By Lemma~\ref{lem.conc} (with $b=1$), $\pr(\vol(G_p) < \tfrac23 \vol(G)p) \leq \tfrac{\eps}6$ (if $cn$ is sufficiently large).
Then we have
\[  \pr(\cB_1) \leq \pr(\cB_3) + \pr(\vol(G_p) < \tfrac23 \vol(G)p) \leq \tfrac{\eps}3,\]
as required.
\end{proof}
We now start to prove that the inequality~(\ref{eqn.B2}) holds. The proof proceeds by considering first degree tax in Claim (\ref{eqn.E3}), and then edge contribution in Claim (\ref{eqn.eintsmall2}).

\bigskip

\noindent\emph{Degree tax} \:
Let $\cB_4$ be the event that for some $A \subseteq V$ with $\vol_{G}(A) \geq  \frac{\eta}2 \vol(G)$ we have $\vol_{G_p}(A) < (1-\tfrac{\eps}{9}) \vol_G(A)p$. 
We claim that 
\begin{equation}\label{eqn.E3}
 \pr(\cB_4) \leq \tfrac{\eps}{12}.\end{equation}

\begin{proof}[Proof of Claim~(\ref{eqn.E3})]
Let $A \subseteq V$ have $\vol_{G}(A) \geq  \eta\, \vol(G)$. Then $\vol_{G_p}(A)$ has mean $\vol_G(A)p \geq \tfrac{\eta}2 \vol(G)p \geq \tfrac{\eta}2 cn$.  Hence, by Lemma~\ref{lem.conc} (with $b=2$) and a union bound, we obtain~(\ref{eqn.E3}) if $c \geq K/\eps$ for a sufficiently large constant~$K$.
\end{proof}

Let $\cE_1$ be the event that $\vol(G_p) \leq \big((1\!-\!\tfrac{\eps}{9})/(1\!-\!\tfrac{\eps}{8})\big)\, \vol(G)p$.
By Lemma~\ref{lem.conc} (with $b=1$), $\pr(\overline{\cE}_1) < \tfrac{\eps}{12}$ (if $cn$ is sufficiently large).
Suppose that $\cE_1 \land \overline{\cB_4}$ holds: then for each partition $\cA=(A_1,\ldots,A_k) \in \cQ$
\begin{eqnarray*}
q^D_{\cA}(G_p) 
& \geq &
\frac{\sum_i  \big( (1-\tfrac{\eps}9) \vol_G(A_i) p \big)^2}{(((1 - \tfrac{\eps}9)/(1-\tfrac{\eps}8)) \vol(G)p)^2} \\
& = &
 \big(1- \tfrac{\eps}8 \big)^2  q^D_{\cA}(G) \; > \;
(1-\tfrac{\eps}4) q^D_{\cA}(G) \; \geq \; q^D_{\cA}(G) - \tfrac{\eps}4.
\end{eqnarray*}
Thus
\begin{equation} 
\label{eqn.qDbig}
 \pr(\exists\, \cA \in \cQ : q^D_{\cA}(G_p) \leq q^D_{\cA}(G) - \tfrac{\eps}4) \leq 
\pr(\overline{\cE_1}) + \pr(\cB_4) \leq \tfrac{\eps}6 . 
\end{equation}
\smallskip

\noindent\emph{Edge contribution} \:
Let $\cE_2$ be the event that for each partition $\cA \in \cQ$ we have $\eint(G_p) \leq \max \{ (1+\eta) \eint(G) p , \, \eta \, e(G)p \}$. 
We claim that (if $c$ is sufficiently large) 
\begin{equation} \label{eqn.eintsmall2}
\pr(\overline{\cE}_2) \leq \tfrac{\eps}{12}.\end{equation}
\begin{proof}[Proof of Claim~(\ref{eqn.eintsmall2})]
Each partition $\cA \in \cQ$ has at most $\frac2{\eta}$ parts, so
$|\cQ| \leq (1+ \frac2{\eta})^n$.
Let $\cB_5$ be the event that, for some partition $\cA \in \cQ$ such that $\eint(G) < \tfrac{\eta}2 e(G)$, we have $\eint(G_p) \geq \eta e(G)p$.  
For each given such partition $\cA$, $\eint(G_p)$ is stochastically at most a binomial random variable with mean $\tfrac{\eta}2 e(G)p$, so by Lemma~\ref{lem.conc} (with $b=1$) and a union bound we obtain $\pr(\cB_5) \leq \tfrac{\eps}{24}$, if $c > K \eps^{-1} \log \eps^{-1}$ for a sufficiently large constant $K$. 

Let $\cB_6$ be the event that, for some partition $\cA \in \cQ$ such that $\eint(G) \geq \tfrac{\eta}2 e(G)$, we have $\eint(G_p) > (1+ \eta) \eint(G)p$. For each given such partition $\cA$, $\eint(G_p)$ is a binomial random variable with mean at least $\tfrac{\eta}2 e(G)p \geq \tfrac{\eta}2 cn$, so by Lemma~\ref{lem.conc} (with $b=1$) and a union bound we obtain $\pr(\cB_6) \leq \tfrac{\eps}{24}$, if $c > K \eps^{-3} \log \eps^{-1}$ for a sufficiently large constant $K$. 

Now consider any partition $\cA \in \cQ$.  If $\cB_5$ fails and $\eint(G_p) \geq \eta e(G)p$ then $\eint(G) \geq \tfrac{\eta}2 e(G)$; and so, if also $\cB_6$ fails, then $\eint(G_p) \leq (1+ \eta) \eint(G)p$.  Thus if both $\cB_5$ and $\cB_6$ fail then $\cE_2$ holds, so
\[ \pr(\overline{\cE}_2) \leq \pr(\cB_5) + \pr(\cB_6)  \leq \tfrac{\eps}{12}, \]
so the Claim~(\ref{eqn.eintsmall2}) holds, as required.
\end{proof}

Having proved the claims (\ref{eqn.E3}) and~(\ref{eqn.eintsmall2}), we may now complete the proof of~(\ref{eqn.B2}).

\begin{proof}[Proof of~(\ref{eqn.B2})]
Let $\cE_3$ be the event that $e(G_p) \geq (1+ \eta)^{-1} e(G)p$.  By Lemma~\ref{lem.conc} (with $b=1$), $\pr(\overline{\cE}_3) \leq \tfrac{\eps}{12}$ if $c > K \eps^{-2} \log \eps^{-1}$ for a sufficiently large constant $K$. 
If $\cE_2$ and $\cE_3$ hold, then for each partition $\cA \in \cQ$ we have
\begin{eqnarray*} 
  q^E_{\cA}(G_p)
& \leq &
  \frac{\max \{(1+\eta) \eint(G) p, \, \eta \, e(G)p \}}{(1+\eta)^{-1} e(G)p}\\
& = &
  \max \{(1+\eta)^2 q^E_{\cA}(G),\, \eta+\eta^2 \}\\
& < &
  q^E_{\cA}(G) + 2 \eta + \eta^2 \;\; < \;\;  q^E_{\cA}(G) + \tfrac{\eps}4.
\end{eqnarray*}
Hence, using also~(\ref{eqn.eintsmall2}),
if $c > K \eps^{-3} \log \eps^{-1}$ for a sufficiently large constant~$K$,
\begin{equation} \label{eqn.qEsmall} \pr(\exists \cA \in \cQ : q^E_{\cA}(G_p) \geq q^E_{\cA}(G) + \tfrac{\eps}4) \leq \pr(\overline{\cE}_2) + \pr(\overline{\cE}_3)  \leq \tfrac{\eps}6. 
\end{equation}
But now, by~(\ref{eqn.qDbig}) and~(\ref{eqn.qEsmall}), 
\begin{eqnarray*}
  \pr(\cB_2)
&=&
   \pr(\exists \cA\!\in\!\cQ : q_{\cA}(G_p) > q_{\cA}(G)\!+\!\tfrac{\eps}{2})\\
& \leq &
  \pr(\exists \cA\!\in\!\cQ: q^D_{\cA}(G_p)\!<\!q^D_{\cA}(G)\!-\!\tfrac{\eps}4) +
\pr(\exists \cA\!\in\!\cQ : q^E_{\cA}(G_p)\!\geq\!q^E_{\cA}(G)\!+\!\tfrac{\eps}4)\leq \tfrac{\eps}3. 
\end{eqnarray*}
This completes the proof of~(\ref{eqn.B2}).  
\end{proof}
We have now completed the proofs of~(\ref{eqn.a}),~(\ref{eqn.B1}) and~(\ref{eqn.B2}); so it remains only to prove the containment~(\ref{eqn.b}). 
\begin{proof}[Proof of~(\ref{eqn.b})]
Recall from Lemma~\ref{lem.nosmall2} that $q_{\cA'_0}(G_p) \geq q_{\cA_0}(G_p) - 2 \eta$, where $\cA'_0=\cA(G_p,\eta,\cA_0)$.  Hence, if we let $\cE_4$ be the event that there is a partition $\cA_1$ which is $\eta$-fat for $G_p$ and satisfies 
$q_{\cA_1}(G) < q_{\cA_1}(G_p) - \eps + 2\eta$ then $\cE_0 \subseteq \cE_4$.  
Let $\cE_5$ be the event that there is a partition $\cA_2 \in Q$ (that is, $\cA_2$ is $\frac{\eta}{2}$-fat for $G$) which satisfies 
$q_{\cA_2}(G) < q_{\cA_2}(G_p) - \eps + 2\eta$.  Then $\cE_4 \subseteq \cB_1 \cup \cE_5$, by the definition of the event $\cB_1$. 
Also, by our choice of $\eta$, if $\cE_5$ holds then there exists $\cA \in Q$ such that $q_{\cA}(G) < q_{\cA}(G_p) - \frac{\eps}2$, i.e. $\cB_2$ holds; and thus $\cE_5 \subseteq \cB_2$.  
Summing up, we have $\cE_0 \subseteq \cE_4 \subseteq \cB_1 \cup {\cE_5} \subseteq \cB_1 \cup \cB_2\,$; so~(\ref{eqn.b}) holds, as required. \end{proof}\label{end.proofthm2}

We have now completed the proofs of \eqref{eqn.a}-\eqref{eqn.B2}, and thus completed the proof of Theorem~\ref{thm.moddiff}.\end{proof} 

In the proof of Theorem~\ref{thm.moddiff} we took $c(\eps)=\Theta(\eps^{-3} \log \eps^{-1})$. The following example shows that, for the conclusion (a) in Theorem~\ref{thm.moddiff} to hold, $c(\eps)$ must be at least $\Omega(\eps^{-2})$.

\begin{example}\label{ex.moddiffconstant}
Recall from~\cite[Theorem 1.3]{ERmod} that there is a constant $a>0$ such that, for each $c\geq 1$ whp $\q(G_{n, c/n}) >a/\sqrt{c}$.
Let $\eps >0$ and let $c = c(\eps) = a^2/\eps^2$ (so $\eps= a/\sqrt{c} $).  Let $n \geq c+1$, let $G$ be~$K_n$ (so $\q(G)=0$) and let $p=(c+1)/n$.  Then the expected average degree in $G_p$ is $(n-1)p \geq c$, but whp $\q(G_p) > \q(G)+ \eps$.\end{example}

\needspace{3\baselineskip}
\subsection{A $k$-part analogue to Theorem~\ref{thm.moddiff}}\label{subsec.moddiff_k}
As in Section~\ref{subsec.obsmod_k}, recall that $\qk(G)=\max_{|\cA|\leq k} q_\cA(G)$. We present here a variant of Theorem~\ref{thm.moddiff}(a) where we replace each instance of $\q$ by $\qk$. In Section~\ref{sec.further_and_relation} it was mentioned in the discussion concerning the stochastic block model following~(\ref{eq.SBM_ub}) that the case $k=2$ of this variant will be used in Remark~\ref{rem.bootstrap} to obtain the relevant upper bound.
\begin{proposition}\label{prop.moddiff_k}
For each $\eps>0$, there is a $c=c(\eps)$ such that the following holds.  Let the graph $G$ and probability $p$ satisfy $e(G)p/v(G) \geq c$, and let $k \geq 2$. Then with probability $\geq 1-\eps$ the 
random graph $G_p$ satisfies $|\qk(G_p) - \qk(G)| < \eps$.
 \end{proposition}

\begin{proof}
We may follow the proof of Theorem~\ref{thm.moddiff}(a), with a few natural small changes which we detail below. Note that since we do not prove an analogue of Theorem~\ref{thm.moddiff}(b), we need only prove analogues of~\eqref{eqn.a}, \eqref{eqn.B1} and \eqref{eqn.B2}, i.e. not~\eqref{eqn.b}.

For partitions with at most $k$ parts we define analogues of the events $\cB_0$ and $\cB_2$ (the event $\cB_1$ is unchanged, and no analogue of $\cE_0$ is needed since we do not prove the analogue of \eqref{eqn.b}). Let $\cB_0^{(k)}$ be the event $\{\qk(G_p) < \qk(G) - \eps\}$. 
By Theorem~\ref{thm.moddiff} (b) (with $\eps$ replaced by $\eps/3$) applied to an optimal partition $\cA$ with at most $k$ parts (and noting that $\cA'$ has at most $k$ parts), we have $\pr(\cB_0^{(k)}) \leq \eps/3$.

As before let $\eta= \tfrac{\eps}{9}$. Let $\cQ^{(k)}$ be the set of partitions which are $\frac{\eta}2$-fat for $G$ and have at most $k$-parts; and let $\cB_2^{(k)}$ be the event that there is a partition $\cA \in \cQ^{(k)}$ such that $q_{\cA}(G_p) > q_{\cA}(G)+\tfrac{\eps}{2}$. 
We must establish the analogues of the statements~(\ref{eqn.a}), \eqref{eqn.B1} and~\eqref{eqn.B2}.

\begin{proof}[Proof of analogue of~(\ref{eqn.a})]
We want to show that
$ \{ \qk(G_p) > \qk(G)+\eps\}  \subseteq \cB_1 \cup \cB_2^{(k)}$.
As before by Lemma~\ref{lem.nosmall2} (and since $2 \eta \leq \eps/2$),
\begin{eqnarray*}
\{ \qk(G_p) > \qk(G)+\eps\}
& \subseteq &
  \{\exists\, \cA: \: |\cA|\leq k, \: \cA \mbox{ is } \eta\!-\!\mbox{fat for} \, G_p, \: q_{\cA}(G_p) > \qk(G)+\tfrac{\eps}{2}\} \\
& \subseteq &
  \{\exists\, \cA: \: |\cA|\leq k, \: \cA \mbox{ is } \eta\!-\!\mbox{fat for} \, G_p, \: q_{\cA}(G_p) > q_{\cA}(G)+\tfrac{\eps}{2}\} \\
& \subseteq &
  \{\exists\, \cA \not\in \cQ^{(k)} : \: |\cA|\leq k, \: \cA \mbox{ is } \eta\!-\!\mbox{fat for} \, G_p \} \cup \cB_2^{(k)} \\
& \subseteq & \cB_1 \cup \cB_2^{(k)},\end{eqnarray*}
as required.
\end{proof}

Since the event $\cB_1$ is unchanged, the statement of \eqref{eqn.B1} is unchanged and thus $\pr(\cB_1) < \eps/3$. The analogue of~\eqref{eqn.B2} is to prove that $\pr(\cB_2^{(k)})<\eps/3$. But $\cB_2^{(k)}\subseteq \cB_2$ and thus $\pr(\cB_2^{(k)}) \leq \pr(\cB_2) <\eps/3$ by (the original)~\eqref{eqn.B2}. 
We have now completed the proof of Proposition~\ref{prop.moddiff_k}.
\end{proof}

\begin{remark}\label{rem.bootstrap}
We may now complete the discussion of the bootstrapping example from Section~\ref{subsec.bootstrap}.  We give details for how to show that equality holds in~(\ref{eq.SBM_ub}) when $p, q = \omega(n^{-1})$. This will follow from spectral results on denser graphs, robustness of modularity and Proposition~\ref{prop.moddiff_k}.

Recall that $\q(G)\leq \bar{\lambda}$ and $\qtwo(G)\leq \bar{\lambda}/2$ where $\bar{\lambda}$ is the spectral gap of $G$, see Lemma 6.1 of~\cite{ERmod}. Now let $a> b>0$ satisfy the condition $\sqrt{a} - \sqrt{b}>\sqrt{2}$. Then a result from~\cite{deng2021strong} 
says that, for $p=n^{-1} a \log n$ and $q=n^{-1} b \log n$,
we have $\bar{\lambda}=(a-b)/(a+b)+o(1)$ whp; and thus
\begin{equation}\label{eqn.planted1}
\qtwo(G_{n, 2, p, q}) = \frac{p-q}{2(p+q)}+o(1) \;\;\; \mbox{whp}.
\end{equation}
We can use Proposition~\ref{prop.moddiff_k} to bootstrap~(\ref{eqn.planted1}) to sparser $p$ and $q$ (and with no condition corresponding to $\sqrt{a} - \sqrt{b}>\sqrt{2}$).
Let $\alpha>\beta>0$; let $\omega=\omega(n) \to \infty$ (arbitrarily slowly) as $n \to \infty$, with $\omega(n) \leq \log n$; and let $p= \alpha\, \omega /n$ and $q= \beta\, \omega /n$. Let us show that still~\eqref{eqn.planted1} holds.
Let $p^+=c \alpha (\log n)/n$ and $q^+ = c \beta (\log n)/n$, for some constant $c  \geq 1$ such that $\sqrt{c\alpha}-\sqrt{c\beta}>\sqrt{2}$, so whp by~\eqref{eqn.planted1}
\[ \qtwo(G_{n, 2, p^+,q^+}) = \frac{p^+-q^+}{2(p^++q^+)} +o(1)  =\frac{p-q}{2(p+q)} +o(1).\]
Also, whp $e(G_{n, 2, p^+,q^+}) = \frac14 (\alpha+\beta) c n \log n +o(n^{2/3})$.  Let $x=\omega/(c \log n)$ (so $0 \leq x \leq 1$).  Then the sampled random graph $(G_{n, 2, p^+,q^+})_x$ satisfies 
$(G_{n, 2, p^+,q^+})_x =_s G_{n,p,q}$, and whp $e((G_{n, 2, p^+,q^+})_x)/n = \frac14 (\alpha+\beta) \omega +o(1) \to \infty$.  Hence by Proposition~\ref{prop.moddiff_k} we have
$\qtwo(G_{n, 2, p, q}) = \qtwo(G_{n, 2, p^+,q^+}) +o(1)$;
whp and we have shown that~\eqref{eqn.planted1} holds with the current choice of $p,q$, as we aimed to do.\end{remark}



\needspace{6\baselineskip}
\section{Robustness of modularity, and closeness and concentration for $\q(G_p)$ and $\q(G_m)$}\label{sec.robust_close_conc}
\subsection{Robustness} 
\label{subsec.robustness}
We shall use deterministic robustness results  relating to two graph operations, namely moving or deleting a set of edges, and concerning two graph parameters, namely the modularity score $q_\cA(G)$ for a given partition $\cA$, and the (maximum) modularity $\q(G)$. We first briefly collect some already known results together with one new bound in Lemma~\ref{lem.robustness} below. For edge deletion the bound on $q_\cA$ is new while the bound on $\q$ follows from Lemma~5.1 in~\cite{ERmod}; and for moving edges the bounds follow from Lemma~5.3 and its proof in the same paper. See also Example~\ref{ex.robustness} below which gives examples showing the bounds in Lemma~\ref{lem.robustness} are asymptotically tight.
\begin{lemma}\label{lem.robustness}
Let $H=(V,E)$ be a graph and let $\cA$ be a partition of $V$.
\begin{itemize}
    \item[(a)]
If $E_0$ is a non-empty proper subset of $E$, and $H'=(V, E \setminus E_0)$, then 
\[ |q_\cA(H)- q_\cA(H')|  \, , \; |\q(H)-\q(H')|< \frac{2 \, |E_0|}{|E|}. \]
\item[(b)]
If $\tilde{E} \neq E$ is a set of edges with $|\tilde{E}|=|E|$, and $\tilde{H}=(V,\tilde{E})$, then
\[ |q_\cA(H)- q_\cA(\tilde{H})| \, , \; |\q(H)-\q(\tilde{H})| < \frac{|E \!\bigtriangleup\! \tilde{E}|}{|E|}. \]
\end{itemize}
\end{lemma}
We give examples to show tightness of Lemma~\ref{lem.robustness} before proceeding with the proof. See the concluding remarks for a related open problem.
\begin{example}\label{ex.robustness}
The examples here show tightness of the lemma - it is not possible to replace the factors 2 and 1 above by smaller constants in any of the four cases. 
    
For the examples let us first recall some simple properties of modularity. Firstly, the placement of any isolated vertices in a partition is irrelevant. Secondly, for any part $A$ in any optimal partition~$\cA$ of a graph, the graph induced by the non-isolated vertices in $A$ is connected~\cite{nphard}. Thirdly, for any leaf vertex $u$ with pendant edge $uv$, the vertices $u$ and $v$ are in the same part in any optimal partition~\cite{nphard, thesis}. 
\begin{itemize}
   \item[(a)] Let $H_1$ be the graph consisting of a star on $m$ edges together with a disjoint edge $e$. Take $E_0=\{e\}$ and thus $H_1'$ is a star on $m$ edges together with two isolated vertices. We consider the bipartition $\cA$ which in $H_1$ places the vertices of the star in one part and the vertices of the disjoint edge in the other. Note that $\cA$ is an optimal partition of $H_1$ and of $H_1'$. 
    Since~$H_1'$ is a star plus isolated vertices, $\q(H_1')=q_\cA(H_1')=0$, and we may calculate 
        \[q_\cA(H_1)=1-\frac{(2m)^2+2^2}{4(m+1)^2}=\frac{2m}{(m+1)^2}.\]
    Noting that $\{e\}/|E|=1/(m+1)$, for this choice of $H_1$ and $H_1'$ we have
        \[  | \q(H_1)- \q(H_1') | = | q_\cA(H_1)-q_\cA(H_1') | = \frac{2m}{(m+1)^2} = \frac{2m}{m+1}\frac{|E_0|}{|E|} \]
    and thus we may get arbitrarily close to the factor 2. 
    \item[(b)] Let $H_2$ be the graph $H_1$ above except that we add an isolated vertex. Thus $H_2$ is the graph consisting of a star with central vertex $u$ on $m$ edges together with a disjoint edge $e$ and an isolated vertex $v$. Let $\tilde{H}_2$ be obtained from $H_2$ by removing $e$ and adding the edge $e'=uv$ : thus $\tilde{H}_2$ is a star on $m+1$ edges together with two isolated vertices. 
        
    We consider the bipartition $\cA$ which places the vertices of the star and the isolated vertex together in one part and the disjoint edge in the other part. Then $\cA$ is an optimal partition of $H_2$ and of $\tilde{H}_2$ by the same argument as in part (a), and $\q(H_2)=\q(H_1)=2m/(m+1)^2$. Since $\tilde{H}_2$ is a star plus isolated vertices $\q(\tilde{H}_2)=q_\cA(\tilde{H}_2)=0$.

    Noting that $|E\bigtriangleup \tilde{E}| / |E| = 2/(m+1)$, for this choice of $H_2$ and $\tilde{H}_2$ we have
        \[  | \q(H_2)-\q(\tilde{H}_2) | = | q_\cA(H_2)-q_\cA(\tilde{H}_2) | = \frac{2m}{(m+1)^2} = \frac{m}{m+1}\frac{|E\bigtriangleup \tilde{E}|}{|E|} \]
    and thus we may get arbitrarily close to the factor 1.
\end{itemize}
\end{example}
To prove the inequality concerning $q_\cA$ in Lemma~\ref{lem.robustness}(a), we will use Lemma~\ref{lem.res_limit_two} bounding the edge contribution minus twice the degree tax. 
\begin{lemma}\label{lem.res_limit_two}
    Let $H$ be a graph on $m\geq 1$ edges and let $\cA$ be a vertex partition. Then
    \[q^E_\cA(H)-2q^D_\cA(H)\geq -1\,.\]
\end{lemma}
\begin{proof}
For each part $A_i \in \cA$, let us write $a_i=\frac1m e(A_i)$ and $b_i=\frac{1}{m}e(A_i, \bar{A_i})$ for the proportion of edges internal to part $A_i$ and proportion between part $A_i$ and the rest of the graph respectively. Note that $\sum_i (a_i + \frac{1}{2}b_i) =1$.  
We also set $a=\sum_i a_i=q^E_\cA(H)$ for the proportion of edges within parts of the partition. Similarly let $b=\frac{1}{2} \sum_i b_i$, so $b$ is the proportion of edges between parts, and note that $b_i \leq b$ for each $i$. Also observe that $a+b=1$.

The degree tax can be written
 \begin{equation}\label{eq.robust_degtax}
        q^D_\cA(H)=\sum_i (a_i+b_i/2)^2 = \sum_i a_i^2 +\sum_i a_i b_i + \tfrac{1}{4}\sum_i b_i^2.
 \end{equation}
We now upper bound each term of the RHS of \eqref{eq.robust_degtax} in turn. First note that $\sum_i a_i=a$ implies $\sum_i a_i^2\leq a^2$.
Next, since each $b_i \leq b$ and $\sum_i a_i = a$ we have $\sum_i a_i b_i \leq ab$. Similarly, each $b_i \leq b$ and $\sum_i b_i = 2b$ implies that $ \sum_i b_i^2 \leq b \,(2b) = 2b^2.$
Thus by these bounds and~\eqref{eq.robust_degtax}
\[q^D_\cA(H)\leq a^2 + ab + \tfrac{1}{2}b^2 = \tfrac{1}{2}+ \tfrac{1}{2}a^2\]
since $a+b=1$. 
Recalling that $q^E_\cA(H)=a$, we get 
\[q^E_\cA(H)-2q^D_\cA(H) \geq a -1 - a^2 = -1 + a(1-a) \geq -1 \]
where the last inequality follows since $0 \leq a \leq 1$.
\end{proof}
\bigskip

\begin{proof}[Proof of Lemma~\ref{lem.robustness}(a), bound on $q_\cA$] We begin by following the proof of Lemma~5.1 in~\cite{ERmod}. As there, let $\alpha=\alpha(\cA) = |E_1|/|E|$
where $E_1$ is the set of edges in $E_0$ which lie within the parts of $\cA$ and $\beta=\beta(\cA)= |E_0 \backslash E_1|/|E|$. That is, $\alpha$ is the proportion of the edges in $E$
which are in $E_0$ and lie within the parts of~$\cA$, and $\beta$ is the proportion which are in $E_0$ and lie between the parts. 

Note that the statement we wish to prove follows from the following two inequalities.
\begin{equation}\label{eq.possible_increase_in_mod}
    q_\cA(H') - q_\cA(H) < 2\alpha + 2\beta 
\end{equation}
and
\begin{equation}\label{eq.possible_decrease_in_mod}
    q_\cA(H) - q_\cA(H') < 2\alpha + \beta \,.
\end{equation}
The first of these, \eqref{eq.possible_increase_in_mod}, is equation (5.2) in~\cite{ERmod} and was proven there, so it remains only to prove~\eqref{eq.possible_decrease_in_mod}.
Directly from the definitions we may calculate the difference in edge contributions, (this is (5.4) in~\cite{ERmod}), 
\begin{equation}\label{eq.difference_in_edge_cont}q^E_\cA(H)-q^E_\cA(H') = \alpha  - (\alpha+\beta)\, q^E_\cA(H')\,.\end{equation}
Note also the following simple bound on the possible increase in the degree tax. Since  $|E \backslash E_0|=(1-\alpha-\beta)|E|$ we have
\begin{equation*}
    q_\cA^D(H) > \frac{1}{4|E|^2}\sum_{A \in \cA} \vol_{H'}(A)^2 = (1-\alpha-\beta)^2 q^D_\cA(H')
\end{equation*}
and so
\[ q_\cA^D(H') - q_\cA^D(H) < 2 (\alpha+\beta) q_\cA^D(H')\,. \]
Thus by~\eqref{eq.difference_in_edge_cont} we have (this is (5.6) in~\cite{ERmod}) 
\begin{equation}\label{eq.possible_decrease_working}
    q_\cA(H) - q_\cA(H') < \alpha\! -\! (\alpha+\beta)q^E_\cA(H')+2(\alpha+\beta)q_\cA^D(H') = \alpha - (\alpha+\beta)(q^E_\cA(H')-2q^D_\cA(H')).
\end{equation}
Now we may apply Lemma~\ref{lem.res_limit_two} to infer that the RHS of~\eqref{eq.possible_decrease_working} is at most $2\alpha+\beta$ which proves~\eqref{eq.possible_decrease_in_mod}.
Finally, by~\ref{eq.possible_increase_in_mod} and~\eqref{eq.possible_decrease_in_mod},
\[ |q_\cA(H)- q_\cA(H')|< 2(\alpha+\beta) =\frac{ 2 \, |E_0|}{|E|},\]
as required.\end{proof}

\needspace{3\baselineskip}
\subsection{Closeness of $\q(G_p)$ and $\q(G_m)$}
\label{subsec.closeness}
Given a graph $G$ and $1 \leq m \leq e(G)$, how close are $\q(G_m)$ and $\q(G_p)$, where $p = m/e(G)$? This question gives rise to Lemma~\ref{lem.GpGm}, which we now state and prove. Note that Corollaries~\ref{cor.obsmod} and~\ref{cor.moddiff} then follow by Lemma~\ref{lem.GpGm} and Theorem~\ref{thm.obsmod} and~\ref{thm.moddiff} respectively.
\begin{lemma}
\label{lem.GpGm}
Let $G$ be a graph, let $1 \leq m \leq e(G)$, and let $p = m/e(G)$. Then for any partition~$\cA$ of~$V(G)$
\begin{equation} \label{eqn.E}
| \E[\q(G_p)] - \E[\q(G_m)] | \, , \; |\E[q_\cA(G_p)] - \E[q_\cA(G_m)] | \leq 2 \sqrt{\tfrac{1-p}{m}}\,.
\end{equation}
Also, we can couple $G_p$ and $G_m$ so that, for any partition $\cA$ of~$V(G)$, if $\omega(m) \to \infty$ as $m \to \infty$ then
\begin{equation} \label{eqn.whp}
| \q(G_p) - \q(G_m)| \; ,\;
| q_\cA(G_p) - q_\cA(G_m)| \; \leq \; \omega(m) \sqrt{\tfrac{1-p}{m}} \;\; \mbox{ whp}\,.
\end{equation}
\end{lemma}

\begin{proof}
Let $X=e(G_p)$, so $X \sim \Bin(e(G),p)$ with mean $m$ and variance $\var(X)=m(1-p)$.  Couple $G_p$ and $G_m$ so that their edge sets are nested: to do this, we may list the edges of $G$ in a uniformly random order, let $G_p$ have the first $X$ edges and let $G_m$ have the first $m$ edges. By Lemma~\ref{lem.robustness}
\begin{equation} \label{eqn.lem}
    | \q(G_p) - \q(G_m)| \leq \tfrac{2\,|X-m|}{m}\,.
\end{equation}
But
\begin{eqnarray*}
  \left(\E[\q(G_p)] - \E[\q(G_m)] \right)^2 & \leq &
  \E[(\q(G_p) - \q(G_m))^2]\\
  & \leq & \tfrac{4}{m^2} \var(X) \; = \; \tfrac{4(1-p)}{m}\,,
\end{eqnarray*}
where we use~\eqref{eqn.lem} for the second inequality; 
and the first inequality in~\eqref{eqn.E} follows. 
The second inequality in~\eqref{eqn.E} may be proved in the just same way.
Also, by Chebyshev's inequality,  for $t >0$
\[ \pr(|X-m|  \geq t) \leq \tfrac{\var(X)}{t^2} = \tfrac{m(1-p)}{t^2}.\]
So, setting $t = \frac12 \omega \sqrt{m(1-p)}$, whp $|X-m| \leq t$.  Hence
\[
| \q(G_p) - \q(G_m)| \leq \tfrac{2t}{m} = \omega \, \sqrt{\tfrac{1-p}{m}} \;\; \mbox{ whp}\,;
\]
and we have proved the first inequality in~(\ref{eqn.whp}).  The second inequality may be proved in just the same way. This  completes the proof of Lemma~\ref{lem.GpGm}.
\end{proof}

\needspace{3\baselineskip}
\subsection{Concentration of $\q(G_p)$ and $\q(G_m)$}\label{subsec.concentration}
We restate Theorem~\ref{thm.bconc} from the introduction, adding also concentration results for the random edge model. Recall that, given a fixed underlying graph $G$, for $0 < m \leq e(G)$ the random graph $G_m$ is obtained by uniformly sampling all $m$-edge subsets of $G$.

For the random graphs $G_{n,m}$ and $G_{n,p}$, Theorem 7.1 of \cite{ERmod} showed that the modularity values are highly concentrated about their expected value. A very similar result is true, with similar proof, for the case of edge sampling. We use the results on $G_m$ to deduce the results on $G_p$.

\begin{theorem}\label{thm.bconcBoth}
\begin{itemize}
\item[(a)] Given graph $G$, a partition $\cA$, and $0 \leq m \leq e(G)$, for each $t>0$ 
\begin{equation*} \pr\Big( \big| \, q_\cA(G_m) - \E[q_\cA(G_m)] \, \big| \geq t \Big) < 2 \, e^{-t^2m/2 } 
\;\mbox{ and }\;
\pr\Big( \big| \q(G_m) - \E[\q(G_m)] \big| \geq t \Big) < 2e^{-t^2m/2}. \end{equation*}
\item[(b)]
There is a constant $\eta>0$ such that for each graph $G$, each partition $\cA=\cA(G)$ and each $0<p<1$ the following holds with $\mu=\mu(G,p) = e(G) p$.  
For each $t \geq 0$ 

\begin{equation*} \pr\Big( \big| \, q_\cA(\Gp) - \E[q_\cA(\Gp)] \, \big| \geq t \Big) < 2 \, e^{-\eta \mu t^2} 
\;\mbox{ and }\;
\pr\Big( \big| \, \q(\Gp) - \E[\q(\Gp)] \, \big| \geq t \Big) < 2 \, e^{-\eta \mu t^2}.\end{equation*}
\end{itemize}
\end{theorem}
The results in Theorem~\ref{thm.bconcBoth} concerning $\q(G_m)$ and $\q(G_p)$ extend Theorem 7.1 of \cite{ERmod}.
As well as the robustness results above, in the proof we  use Lemma~\ref{lem.conc}, and the following concentration result from~\cite{bbddiffs} Theorem 7.4 (see also Example 7.3).
\needspace{4\baselineskip}
\begin{lemma}
\label{lem.setconc}
 Let $A$ be a finite set, let $a$ be an integer such that $0\leq a \leq |A|$, and consider the set $\binom{A}{a}$ of all $a$-element subsets of $A$.
  Suppose that the function $f:\binom{A}{a}\rightarrow \R$ satisfies $|f(S)-f(T)|\leq c$ whenever $|S \triangle T|=2$  (i.e.\ the $a$-element subsets $S$ and $T$ are minimally different). If the  random variable $X$ is uniformly distributed over $\binom{A}{a}$, then 
 \begin{equation*} \pr \left( \Abs{f(X)-\E[f(X)]} \ge t \right)  \le 
  2 e^{-2t^2/a c^2}.\end{equation*} 
\end{lemma}

\bigskip
\begin{proof}[Proof of Theorem~\ref{thm.bconcBoth} part (a)] 
By Lemma~\ref{lem.robustness}(b), if $E(H)$ and $E(\tilde{H})$ are both of size $m$ and are minimally different then $|q_\cA(H)-q_\cA(\tilde{H})|$, $|\q(H)-\q(\tilde{H})| < 2/m$. Hence, Lemma~\ref{lem.setconc} with $A = E(G)$, $a=m$ and taking $c=2/m$ for $q_\cA$ and  for $\q$ immediately yields part (a) of Theorem~\ref{thm.bconcBoth}.
\end{proof}

\begin{proof}[Proof of Theorem~\ref{thm.bconcBoth} part (b)] 
Let $M=e(G_p)$ and let $\mu=\mu(n,p) = \E[M] = e(G) p$. 
We will first show the more detailed statement that,
for each $t \geq 42/\sqrt{\mu}$, we have
\begin{equation} \label{eqn.showconc}
\pr \big( \big|q_\cA(G_p)-\E[q_\cA(G_p)] \big| \geq t \big) \leq 6 e^{-t^2 \mu/150},
\end{equation}
from which the result for $q_\cA$ in part~(b) of the theorem follows. 
We show the proof in detail for concentration of $q_\cA$ then indicate the few differences to be made in showing the result for~$\q$.

For any graph and any partition the modularity score lies in the interval~$[-1/2,1)$, see~\cite{nphard}; and hence we may assume that $0 \leq t < 3/2$. Define the event $\mathcal{E} = \{ M > 2\mu /3 \}$ and let $\cE^c$ denote the complement of~$\cE$,
\begin{eqnarray}\label{eqn.splitintotwo}
&&\!\!\!\!\!\!\!\!\pr \big( \big|q_\cA(G_p)\!-\!\E[q_\cA(G_p)] \big| \geq t \big) \\
\notag&&\leq 
\pr\big(\big(\big|q_\cA(G_p)\!-\!\E[q_\cA(G_p)|M] \big| \geq \tfrac{t}{2} \big)  \land \cE \big) + \pr\big( \big|\E[q_\cA(G_p)|M]\!-\!\E[q_\cA(G_p)] \big|\geq \tfrac{t}{2}  \big) \land \cE \big) + \pr(\cE^c).
\end{eqnarray}
The proof proceeds by bounding separately the terms on the right in \eqref{eqn.splitintotwo}. Firstly, by using part~(a) of the theorem and conditioning on $M=m$ where $m >2\mu/3$, we have a bound on the first term of~\eqref{eqn.splitintotwo}
\begin{eqnarray}
\notag
\pr\Big( \big(\, \big|\, q_\cA(G_p)-\E[\q(G_p)|M] \,\big| \geq \tfrac{t}{2} \big) \land \cE \Big)
&\leq& 
2 \exp ( - \tfrac12 (\tfrac{t}{2})^2 (\tfrac{2 \mu}{3}) ) = 2\exp(-t^2 \mu/ 12 ).
\end{eqnarray}
The third term is also straight forward to bound. By Lemma~\ref{lem.conc} and since $t < 3/2$, 
\[ \pr(\cE^c) \leq 2 \exp( - \tfrac13 (\tfrac 13)^2 \mu) = 2 \exp(-\mu/27)  <  2 \exp(- (4/243) t^2\mu).
\]
It now remains to bound the second term of~\eqref{eqn.splitintotwo}. Let $G'_{p}$ be a random subgraph of $G$ obtained by keeping each edge with probability $p$, independently of $G_p$, and let $M'=e(G_p')$. By coupling and by Lemma~\ref{lem.robustness}, for $0< m \leq e(G)$
\begin{equation*}
\Big| \E\big[q_\cA(G_p)|M=m] - \E[q_\cA(G'_p)|M'=m'] \Big| \leq \frac{2|m-m'|}{m} .\end{equation*} 
For any $x$ and any random variable $Y$ with finite mean $\big| x - \E[Y] \big| \leq \E_Y[|x-Y|]$ and thus for $0<m \leq e(G)$, 
\begin{eqnarray}
\notag
\big|\E [q_\cA(G_p)|M=m] -\E(q_\cA(G'_p))\big| & \leq  & 
\E_{M'}\big[ \big|\E [q_\cA(G_p)|M=m]-\E[q_\cA(G_p')|M'] \big| \big] \\
\notag
& \leq & (2/m) \, \E_{M'} [|m-M'|] \\
\label{eqn.threeonm}& \leq & (2/m)\, \big(|m-\mu|+ \E[|M'-\mu|] \big).
\end{eqnarray}
Note $\E[|M'-\mu|]  \leq \sqrt{\E[(M' -\mu)^2]} \leq \sqrt{\mu}$ and so
\begin{eqnarray*}
\pr\Big( \big(\big|\E[q_\cA(G_p)|M] - \E[q_\cA(G_p)] \big| \geq \tfrac{t}{2} \big) \land \cE \Big) 
\notag
& \leq & \pr\Big( \big( | M-\mu| +\sqrt{\mu} \geq \tfrac{tM} 4 \big) \land \cE \Big)\notag \\
&\leq& 
\pr\Big( | M-\mu| \geq \tfrac{t \mu} 6-\sqrt{\mu}  \Big) \notag \\ 
&\leq& 
\pr\Big( | M-\mu| \geq \tfrac{t \mu} 7 \Big) 
\end{eqnarray*} 
since we assumed that $t \geq 42 \mu^{-1/2}$ and so $t\mu/6-\sqrt{\mu} \geq t\mu/7$. By Lemma~\ref{lem.conc} 
\[\pr\Big( | M-\mu| \geq t \mu /7 \Big)  \leq 2\exp\Big( - \tfrac13 (t/7)^2 \mu  \Big) = 2 \exp\big(- t^2 \mu/147\big),\]
and thus we have a bound for the second term
\begin{equation*} \pr\Big( \big(\big|\E[q_\cA(G_p)|M] - \E[q_\cA(G_p)] \big| \geq \tfrac{t}{2} \big) \land \cE \Big) \leq 2\exp\big( - t^2\mu /147 \big),\end{equation*}
which completes the proof of~(\ref{eqn.showconc}). It is now possible to choose $\eta$ small enough that the inequality holds for all $t \geq 0$ and we may take $2$ as the coefficient of the exponential - the details of this calculation appear for example in the last part of the proof of Theorem~7.1 of \cite{ERmod}.  This completes the proof of the first half of part (b).

To prove the corresponding result for $\q$ the same technique may be used; in particular we bound the probability that $|\q(G_p)-\E[\q(G_p)]|$ is large by considering the sum of three terms as in \eqref{eqn.splitintotwo}. The first and  second term can then be handled in the same way. For the third term we can get a slightly better bound by noticing that since $\q(H) \in [0,1)$ for any graph $H$ we may assume~$t < 1$ instead of $t < 3/2$.
\end{proof}

\needspace{6\baselineskip}
\section{Estimating modularity by sampling a fixed number of vertices: proof of Theorem~\ref{thm.modest}}
\label{sec.estim}

A graph parameter $f$ is \emph{estimable} (or \emph{testable}), see~\cite{BCLSVjournal} and \cite{lovasz2012large}, if for every $\eps>0$ there is a positive integer $k = k(\eps)$ such that if $G$ is a graph with at least $k$ vertices, then for $X$ a uniformly random k-subset of $V(G)$ we have $\pr(|f(G)-f(G[X])| >\eps)<\eps$. 
We shall see here that the modularity value~$\q(G)$ is estimable for dense graphs but not more generally. For convenience we restate Theorem~\ref{thm.modest} from the introduction.

\begin{theorem*}[restatement of Theorem~\ref{thm.modest}] 
(a) For fixed $\rho$ with $0<\rho<1$, modularity is estimable for graphs with density at least~$\rho$. (b) For any given function $\rho(n) = o(1)$, modularity is not estimable for $n$-vertex graphs with density at least~$\rho(n)$.
\end{theorem*}

We need some definitions.
If $G$ is a graph and $S, T \subseteq V(G)$ then $e_G(S,T)$ is the number of ordered pairs $(s,t) \in S \times T$ such that there is an edge in $G$ between $s$ and $t$. If $G$ and $G'$ are graphs with the same vertex set $V$, then the \emph{cut distance} $d_\square(G,G') := |V|^{-2} \max_{S,T}|e_G(S,T)- e_{G'}(S,T)|$.
Given a graph $G$ and $b \in \NN$, we let $G(b)$ denote the $b$-\emph{blow-up} of $G$, where each vertex of $G$ is replaced by $b$ independent copies of itself; and thus $v(G(b))=b\,v(G)$ and $e(G(b))= b^2\, e(G)$.

\begin{example}
$\q(C_5(2)) > \q(C_5)$.  The unique form of an optimal partition of $C_5$ is into one path $P_2$ and one $P_3$, with modularity 
$\frac35 - \frac{4^2 + 6^2}{10^2} = \frac35 - \frac{52}{100} = \frac 2{25}$. For $C_5(2)$ we can balance the partition, into two copies of $K_{2,3}$, with modularity score $\frac{12}{20} - \frac12 = \frac1{10} > \frac2{25}$.
\end{example}
\bigskip
\begin{proof}[Proof of Theorem~\ref{thm.modest}(a)]
We shall use Theorem~15.1 of \cite{lovasz2012large}, which says that a graph parameter~$f(G)$ is estimable if and only if the following three conditions hold.
\begin{itemize}
    \item[(i)] If $G_n$ and $G'_n$ are graphs on the same vertex set and $d_\square(G_n, G'_n)\!\rightarrow\! 0$ then $f(G_n)\!-\!f(G'_n)\!\rightarrow\!~0$. 
    \item[(ii)] For every graph $G$, $f(G(b))$ has a limit as $b \rightarrow \infty$.
    \item[(iii)] $f(G\!\cup\!K_1)-f(G) \rightarrow 0$ if $v(G)\rightarrow \infty$ (where $G\!\cup\!K_1$ denotes the graph obtained from $G$ by adding a single isolated vertex).
\end{itemize}
Observe that always $\q(G\!\cup\!K_1)=\q(G)$, so we need be concerned here only about conditions (i) and~(ii).  We shall show that condition (ii) concerning blow-ups always holds.  After that, we shall show that condition (i) concerning cut distances holds, as long as the graphs are suitably dense.  Finally we give examples for sparse graphs which show that in this case modularity is not estimable.
\paragraph*{Blow-ups of a graph : condition (ii)}
Recall that $G(b)$ is the $b$-\emph{blow-up} of $G$.
Observe that always $\q(G(b)) \geq \q(G)$.  For if $\cA$ is an optimal partition for $G$, with $k$ parts, then the natural corresponding $k$-part partition for $G(b)$ has modularity score $\q(G)$.  Thus also $\q(G(jb)) \geq \q(G(b))$ for every $j \in \NN$. \,

Let $G$ be a (fixed) graph.  We need to show that $\q(G(b))$ tends to a limit as $b \to \infty$. Let $q^{**}(G)$ be $\sup_b \q(G(b))$ where the sup is over all $b \in \NN$. We shall see that in fact
\begin{equation} \label{eqn.qstarstar}
\q(G(b)) \to q^{**}(G) \;\;\; \mbox{ as } \; b \to \infty.
\end{equation}
If $G$ has no edges then $\q(G(b))=0$ for all $b \in \NN$; so we may assume that $e(G) \geq 1$.
Let $a \in \NN$, let $j \in \NN$ and let $b \in \NN$ satisfy $(j-1)a \leq b < ja$.  Then
\[ \frac{e(G(ja)) - e(G(b))}{e(G(ja))} = \frac{j^2a^2 - b^2}{j^2a^2} \leq \frac{j^2 - (j-1)^2}{j^2} < \frac{2}{j}. \]
Hence by a robustness result, Lemma 5.1 of~\cite{ERmod}, (see also Lemma~\ref{lem.robustness} in this paper) we have that
$| \q(G(b)) - \q(G(ja))| < \tfrac4{j}$;
and so
\[ \q(G(b)) > \q(G(ja)) - \tfrac4{j}. \]
Let $\eps>0$.  Let $a \in \NN$ be such that $\q(G(a)) \geq q^{**}(G) - \eps/2$.
Then, for all $b \in \NN$ such that $b \geq 8a/\eps$, letting $j \in \NN$ be such that $(j-1)a \leq b < ja$\, (so $j > 8/\eps$ and thus $4/j < \eps/2$) we have
\[ \q(G(b)) \geq \q(G(ja)) - \eps/2  \geq \q(G(a)) - \eps/2 \geq q^{**}(G) - \eps \,. \]
Now~(\ref{eqn.qstarstar}) follows, as required.

\paragraph*{Cut distance and modularity : condition (i) for dense graphs}

Fix $\rho$ with $0<\rho < 1$.  A graph $G$ is called $\rho$-\emph{dense} if $e(G) \geq \rho \, v(G)^2/2$.
Let $0<\eps<1$.  We want to show that there exists $\delta>0$ such that for all $\rho$-dense graphs $G$, $G'$ with the same vertex set and such that $d_\square(G, G') \leq \delta$ we have $|\q(G)-\q(G')| \leq \eps$.  With foresight, we shall take $\delta = \tfrac{\rho \, \eps}{16 \,+\,4/\eps}$. 

By Lemma~1 in~\cite{dinh2011finding} there is a $k \leq k_0 = \lceil2/\eps\rceil$ and a partition $\cA = (A_1,\ldots,A_k)$ for $G$ such that $q_\cA(G) \geq \q(G)-\eps/2$. It suffices to show that
\begin{equation} \label{eqn.qshow}
q_\cA(G') \geq q_\cA(G)- \eps/2
\end{equation}
(since then $\q(G') \geq q_\cA(G') \geq \q(G)-\eps$, and we may similarly deduce that $\q(G) \geq  \q(G')-\eps$).
To prove~(\ref{eqn.qshow}) we first consider the edge contribution $q^E_\cA$ then the degree tax $q^D_\cA$.  Let $n=v(G)=v(G')$.\\
\smallskip

\noindent\emph{Edge contribution} $q^E_\cA$. \;
Note that
$2 |e(G)-e(G')| \leq d_\square(G, G')\, n^2 \leq \delta n^2$, so $e(G') \leq e(G)+ \frac12 \delta n^2$; and similarly
$|\inter_G(A_i)- \inter_{G'}(A_i)| \leq \frac12 \delta n^2$, so $\inter_G(A_i) \geq \inter_{G'}(A_i) - \frac12 \delta n^2$. Thus
\begin{eqnarray*}
q^E_\cA(G')-q^E_\cA(G) &=&
\sum_{i=1}^k \big( \frac{\inter_{G'}(A_i)}{e(G')} - \frac{\inter_{G}(A_i)}{e(G)}  \big)\\
& \geq &
\sum_{i=1}^k \big( \frac{\inter_{G}(A_i) - \tfrac12 \delta n^2}{e(G) + \tfrac12 \delta n^2} - \frac{\inter_{G}(A_i)}{e(G)}  \big)\\
&=&
q^E_\cA(G)\, \left( \big(1+ \frac{\delta n^2}{2 e(G)} \big)^{-1} -1 \right) - \sum_{i=1}^k \frac{\delta n^2}{2 e(G)+\delta n^2}.
\end{eqnarray*}
The second term here (minus the sum) is at least $- k\, \tfrac{\delta n^2}{2 e(G)} \geq - k_0 \, \delta/\rho$. Also, since $(1+x)^{-1} \geq 1-x$ for $x \geq 0$, the first term is at least
$- q^E_\cA(G) \, \tfrac{\delta n^2}{2 e(G)} \geq - \delta/\rho$.
Hence
\begin{equation} \label{eqn.edgecont}
q^E_\cA(G')-q^E_\cA(G) \geq - (k_0 +1)\, \delta / \rho \geq - (2/\eps +\!2)\, \delta/\rho.
\end{equation}

\smallskip

\noindent\emph{Degree tax} $q^D_\cA$.\;
Since $|\vol(G) - \vol(G')| \leq \delta n^2$ we have
\[ | \vol(G)^2- \vol(G')^2| = (\vol(G)+\vol(G')) \, | \vol(G) - \vol(G')| \leq (\vol(G)+\vol(G'))\, \delta n^2 .\]
Thus, using the last inequality if $\vol(G') \leq \vol(G)$,
\begin{equation} \label{eqn.vol1}
\vol(G')^2 \geq \vol(G)^2 - 2 \delta n^2 \vol(G) = \vol(G)^2 \, (1- 2 \delta n^2/\vol(G)) \geq \vol(G)^2 \, (1- 2 \delta/\rho).
\end{equation}
Also $|\vol_G(A_i) - \vol_{G'}(A_i)| \leq  \delta n^2$ for each $i$.  We claim that
\begin{equation} \label{eqn.vol2}
  \sum_{i=1}^k \vol_{G'}(A_i)^2 - \sum_{i=1}^k \vol_{G}(A_i)^2 \leq 2 \delta n^2 \, \vol(G').
\end{equation}
To show this, let $I= \{ i \in [k]: \vol_{G'}(A_i) \geq \vol_G(A_i)\}$.  For $i \in I$
\[ \vol_{G'}(A_i)^2 - \vol_{G}(A_i)^2 = \big(\vol_{G'}(A_i) + \vol_{G}(A_i)\big) \big(\vol_{G'}(A_i) - \vol_{G}(A_i)\big) \leq 2 \, \vol_{G'}(A_i) \cdot \delta n^2. \]
Thus
\begin{eqnarray*}
  \sum_{i=1}^k \vol_{G'}(A_i)^2 - \sum_{i=1}^k \vol_{G}(A_i)^2
& \leq &
 \sum_{i \in I} \big( \vol_{G'}(A_i)^2 - \vol_{G}(A_i)^2 \big)\\
& \leq &
  2 \delta n^2 \, \sum_{i \in I} \vol_{G'}(A_i)\\
& \leq &
  2 \delta n^2 \, \vol(G'),
\end{eqnarray*}
which completes the proof of~(\ref{eqn.vol2}).
Using~(\ref{eqn.vol2}) and then~(\ref{eqn.vol1}), we find
\begin{eqnarray*}
  q^D_\cA(G') 
& = &
  (\vol (G'))^{-2} \sum_{i=1}^k \vol_{G'}(A_i)^2\\
& \leq &
  (\vol (G'))^{-2} \big( \big( \sum_{i=1}^k  \vol_{G}(A_i)^2 \big) + 2 \delta n^2 \vol (G') \big)\\
& \leq &
  (1- 2 \delta/\rho)^{-1}\, q^D_\cA(G) +    2 \delta n^2 / \vol (G')\\
& \leq &
q^D_\cA(G) + 4 \delta/\rho + 2 \delta/\rho 
\end{eqnarray*}
since $(1-x)^{-1} \leq 1+2x$ for $0 \leq x \leq \frac12$.  Thus
\begin{equation} \label{eqn.degtax}
  q^D_\cA(G')  \leq q^D_\cA(G) + 6 \, \delta/\rho.   
\end{equation}
Putting the results~(\ref{eqn.edgecont}) on $q^E_\cA$ and~(\ref{eqn.degtax}) on $q^D_\cA$ together we have
\[ q_\cA(G') \geq q_\cA(G) - (2/\eps +\!8)\, \delta/\rho \geq q_\cA(G) - \eps/2 \]
by our choice of $\delta$.
Hence (\ref{eqn.qshow}) holds, as required.  This completes the proof of Theorem~\ref{thm.modest}(a), for dense graphs.
\end{proof}
We now give a pair of constructions which will demonstrate that modularity is not estimable.
\begin{example} \label{ex.7}
Let $0 \leq \rho(n) <1$ and let $\rho(n) \to 0$ as $n \to \infty$, arbitrarily slowly, so that in particular $\rho(n) \, n/\log n \to \infty$. Then there are connected graphs $G_n, G'_n$ on vertex set $[n]$ such that $e(G_n), e(G'_n) \geq \rho(n)\, n^2/2$; and as $n \to \infty$, $\q(G_n) \to 1$ and $\q(G'_n) \to 0$, and $e(G_n), e(G'_n) = o(n^2)$ so $d_\square(G_n,G'_n) =o(1)$. 
Thus the graphs $G_n, G'_n$ are `nearly dense' and are close in cut distance $d_{\square}$, but their modularity values are not close.

For $G_n$ we may let $k=k(n) \sim 2 \rho(n) n$, and let $G_n$ be a collection of disjoint $k$-cliques (together with at most $k-1$ $(k+1)$-cliques) joined by edges to form a path.  Then $e(G_n) \sim (n/k) \binom{k}{2} \sim nk/2 \sim \rho(n) n^2$ so $e(G_n) \geq \rho(n) n^2/2$ for $n$ sufficiently large.  Also it is easy to see that the partition~$\cA$ of the vertex set into the cliques satisfies $q_\cA(G_n) \sim 1$.

For $G'_n$ we may consider a binomial random graph $G_{n,\rho}$, which with probability at least $\frac13 +o(1)$ is connected, has at least $\rho(n) n^2/2$ edges, and has modularity at most $\eps(n)$ for a suitable $\eps(n)=o(1)$ by Theorem~1.1(c) of~\cite{ERmod} (or by other results in that paper). 
\end{example}

\begin{proof}[Proof of Theorem~\ref{thm.modest}(b)] 
Observe that if $G$ and $G'$ are two $n$-vertex graphs with at most $\eps n^2/2$ edges then
\[ d_{\square}(G,G') \leq \max \{ 2e(G), 2e(G')\}/n^2 \leq \eps. \]
Also recall condition (i) for estimability in the proof of Theorem~\ref{thm.modest}(a). We may now see that Theorem~\ref{thm.modest}(b) (for graphs that do not have at least constant positive density) follows directly from Example~\ref{ex.7} above.
\end{proof}


\needspace{6\baselineskip}
\section{Under-sampling and overestimating modularity}
\label{sec.undersamp}

When we sample few edges from a graph $H$ it seems that we tend to overestimate its modularity; that is, $\q(H_p)$ tends to be significantly larger than $\q(H)$.
For example, if $H$ is the complete graph $K_n$ and $p=1/n$, then $\q(H)=0$ but $\q(H_p) \to 1$ in probability as $n \to \infty$, see Theorem 1.1 of~\cite{ERmod}. 
Our Theorem~\ref{thm.obsmod} shows that when the expected number $e(H) p$ of edges observed is large, although we may overestimate modularity we are unlikely to underestimate it by much.  In this section we use Theorem~\ref{thm.obsmod} to prove that when the sampling probability $p$ is bounded away from 0, increasing~$p$ is unlikely to increase overestimation by much.

To state the result precisely we give one definition.  For random variables $X$ and $Y$ and $\eps>0$, we say that $X$ $\eps$-\emph{nearly (stochastically) dominates} $Y$ if
\begin{equation} \label{def.nearsd}
    \pr(X \geq t) \geq \pr(Y \geq t+\eps) - \eps \;\;\; \mbox{ for each } t.
\end{equation}
Observe that if say $X$ and $Y$ take values in $[0,1]$ and $X$ $\eps$-nearly dominates $Y$ then $\E[X] > \E[Y] - 2 \eps$, since in this case
\begin{eqnarray*}
\E[X] & \geq &  \int_{0}^{1-\eps} \pr(X \geq t)\, dt \;\; \geq 
\int_{0}^{1-\eps} (\pr(Y \geq t+\eps) - \eps) \, dt\\
& = &
\int_{\eps}^{1} \pr(Y \geq u)\, du - \eps + \eps^2 \;\;
 \geq  \E[Y] -2 \eps + \eps^2\,.
\end{eqnarray*}
We now give the main result of this section.
\begin{proposition}
\label{prop.genH}
Let $0<p_0<1$ and $\eps>0$. Then there exists $c$ such that, for any graph $H$ with at least $c$ edges and any sampling probabilities $p_1, p_2$ with $p_0 \leq p_1 < p_2 \leq 1$, it holds that $\,\q(H_{p_1})$ $\eps$-nearly dominates $\q(H_{p_2})$.
\end{proposition}
The case $p_2=1$ shows that $\pr(\q(H_{p_1}) \geq \q(H) -\eps) \geq 1-\eps$ as in Theorem~\ref{thm.obsmod}, except that now we have the lower bound $p_0>0$ on the sampling probability $p$ (and $c$ depends on $p_0$).
\begin{proof}
By Theorem~\ref{thm.obsmod} there exists $c_0$ such that for all graphs $J$ and all $0 < p \leq 1$ such that $e(J)p \geq c_0$, we have
\[ \pr \big(\q(J_p) > \q(J) - \eps/2 \big) > 1- \eps/2 .\]
Let $A$ be the set of all graphs $J$ with $e(J)p_0 \geq c_0$.  Then 
\begin{equation} \label{eqn.A}
 \pr \big(\q(J_p) > \q(J) - \eps/2 \big) > 1- \eps/2 \;\;\;\; \mbox{ for each } J \in A \; \mbox{ and } \; p_0 \leq p \leq 1 .
\end{equation}
Let $c \in \NN$ be such that \[\pr(\Bin(c,p_0) \geq c_0/p_0) \geq 1-\eps/2.\]
Fix a graph $H$ on vertex set $V$ with $e(H)\geq c$; and note that by the above, for each $p_0 \leq p \leq 1$
\begin{equation} \label{eqn.inAnew} 
\pr(H_{p} \in A) \geq \pr(H_{p_0} \in A) > 1-\eps/2.
\end{equation}
We couple $H_{p_1}$ and $H_{p_2}$ in the natural way.  Let $p=p_1/p_2$, let $Z \sim G_{n,p}$ and $H_{p_2}$ be independent.
For each graph $J$ on vertex set $V$, let $J_Z$ be the graph on $V$ with edge set $E(J) \cap E(Z)$; and observe that $(H_{p_2})_Z \sim H_{p_1}$.  
We have
\begin{eqnarray*}
&&\hspace{-3mm} \pr(\q((H_{p_2})_Z) > \q(H_{p_2}) - \eps)\\
&\! \geq\! &
\sum_{K \in A} \pr\big((H_{p_2}=K) \land (\q(K_Z) > \q(K) - \eps)\big)\\
&\! =\! &
\sum_{K \in A} \pr (H_{p_2}=K)\; \pr (\q(K_p) > \q(K) - \eps) \;\;\; \mbox{ since $Z$ and $H_{p_2}$ are independent, and } K_Z \sim K_p
\\
&\! \geq\! &
(1-\eps/2)
\sum_{K \in A}  \pr(H_{p_2}=K) \;\;\; \mbox{ by~\eqref{eqn.A}} 
\\
&\! \geq\! & (1-\eps/2)^2 > 1-\eps
\;\;\; \mbox{ by~(\ref{eqn.inAnew})}.
\end{eqnarray*}
Hence for every $t$
 \begin{eqnarray*}
 \pr(\q(H_{p_1}) \geq t) & \geq &
 \pr\big( (\q((H_{p_2})_Z) > \q(H_{p_2}) - \eps) \land (\q(H_{p_2}) \geq t +\eps) \big)\\
  & \geq &
 \pr(\q(H_{p_2}) \geq t +\eps) - \pr(\q((H_{p_2})_Z) \leq \q(H_{p_2}) - \eps)\\
 & > &
 \pr(\q(H_{p_2}) \geq t +\eps) - \eps,
  \end{eqnarray*}
so $\q(H_{p_1})$ $\eps$-nearly dominates $\q(H_{p_2})$ as required. \end{proof}

\needspace{6\baselineskip}
\section{Expected modularity when average degree is constant}
\label{sec.expmod}

The modularity of the Erd\H{o}s-R\'enyi (or binomial) random graph $G_{n,p}$ is investigated in~\cite{ERmod}. Given a constant $c>0$ we let $\bar{q}(n,c) =  \E[q^*(G_{n,c/n})]$ for each $n \geq c$.  By Theorem~1.1 of that paper, for $0<c \leq 1$ we have $\bar{q}(n,c) \to 1$ as $n \to \infty$. Let $\bar{q}(c) = 1$ for each $c \in (0,1]$.

\begin{conjecture} [\cite{ERmod}] \label{conj.limit}
For each $c>1$, $\bar{q}(n,c)$ tends to a limit $\bar{q}(c)$ as $n \to \infty$.
\end{conjecture}

It was noted in that paper that if the conjecture holds then the function $\bar{q}(c)$ would be uniformly continuous for $c \in (0,\infty)$. From Theorem~\ref{thm.obsmod} (in the present paper) we shall deduce that also $\bar{q}(c)$ would be non-increasing in~$c$.  We collect results on $\bar{q}(c)$ in the following proposition.

\begin{proposition}\label{prop.expmod}
\noindent
\begin{itemize}
    \item[(i)] for $0 < c \leq 1$, we have $\bar{q}(n,c) \to \bar{q}(c) = 1$ as $n \to \infty$;
    \end{itemize}

    and if Conjecture~\ref{conj.limit} holds then
\begin{itemize}
    \item[(ii)] $0< \bar{q}(c)<1$ for $c > 1$ 
    \item[(iii)] $\bar{q}(c)= \Theta(c^{-\frac12})$ as $c \to \infty$
    \item[(iv)] $\bar{q}(c)$ is (uniformly) continuous for $c \in (0,\infty)$
    \item[(v)] $\bar{q}(c)$ is non-increasing for $c \in (0,\infty)$.
\end{itemize}
\end{proposition}

All but part~$(v)$ of this result comes directly from~\cite{ERmod}:  part~$(i)$ (as we already noted) and part~$(ii)$ are from Theorem~1.1; part $(iii)$ is from Theorem~1.3 and part $(iv)$ is from Lemma~7.4.  Part~$(v)$ will follow immediately from inequality~(\ref{eqn.leq0}) below, so to complete the proof of Proposition~\ref{prop.expmod} it remains only to prove the following lemma.

\begin{lemma} \label{lem.qbar}
Let $0<c<c'$. For each $\eps>0$, there exists $n_0$ such that for all $n \geq n_0$ we have 
$\bar{q}(n,c') - \bar{q}(n,c) < \eps$ ;
and thus
\begin{equation} \label{eqn.leq0}
\limsup_{n \to \infty} \; (\bar{q}(n,c') - \bar{q}(n,c)) \leq 0.
\end{equation}
\end{lemma}

\begin{proof}   By Theorem~\ref{thm.obsmod} there exists $c_0$ such that for all graphs $H$ and all $0 < p \leq 1$ such that $e(H)p \geq c_0$ we have
\[  \pr (\q(H_p) > \q(H) - \eps/4) > 1- \eps/4 \, ,\]
and so
\[ \E[\q(H_p)] \geq (1-\eps/4) (\q(H)-\eps/4) > \q(H) - \eps/2. \]
Let $p=c/c'$.  Let $A$ be the set of all graphs $H$ with $e(H)p \geq c_0$.  Let $n_0  \geq c'$ be sufficiently large that for each $n \geq n_0$ we have $\pr(G_{n, c'/n} \in A) > 1- \eps/2$.  Let $n \geq n_0$. 
Let $X \sim G_{n,c'/n}$ and $Z \sim G_{n,p}$ be independent.  For each graph $H$ on $[n]$, let $H_Z$ be the graph on $[n]$ with edge set $E(H) \cap E(Z)$; and observe that $H_Z \sim H_p$.  Note also that $Y=X_Z$ satisfies $Y \sim G_{n,c/n}$. 
Given a graph $H \in A$, since $H_Z$ and $\indic_{X=H}$ are independent, 
\begin{eqnarray*}
 \E[\q(Y) \indic_{X=H}] 
& = &
  \E[\q(H_Z) \indic_{X=H}] \; = \;  \E[\q(H_p)]\, \pr(X=H)\\
& \geq &
(\q(H)-\eps/2)\, \pr(X=H).
\end{eqnarray*}
Hence
\begin{eqnarray*}
\bar{q}(n,c)  &=&
  \E[\q(Y)] \;\; \geq \;\; \sum_{H \in A} \E[\q(Y) \indic_{X=H}]\\
& \geq &
 \sum_{H \in A}  (\q(H) - \eps/2)\, \pr(X =H)\\
& > &
 \sum_{H \in A}  \q(H)\, \pr(X =H)  - \eps/2 \\
& > &
\E[\q(X)] - \eps \;\;  = \;\; \bar{q}(n,c')  - \eps
\end{eqnarray*}
since $\pr(X \not\in A) < \eps/2$.
Thus $\bar{q}(n,c') - \bar{q}(n,c) < \eps$ for each $n \geq n_0$.
It follows that
\[\limsup_{n \to \infty} \; (\bar{q}(n,c') - \bar{q}(n,c)) \leq \eps \,;\]
and since this holds for each $\eps>0$, the inequality~(\ref{eqn.leq0}) follows.
\end{proof}


\needspace{6\baselineskip}
\section{Modularity and edge-sampling on weighted networks}\label{sec.wmod_new}
In network applications it can be useful to consider graphs in which the edges have weights.
Following the notation of~\cite{chung1997spectral}, let $V$ be a non-empty vertex set, and let $w: V\times V\to \R$ satisfy $w(u,v)=w(v,u) \geq 0$ for all vertices $u$ and $v$.  For simplicity, let us assume that $w(v,v)=0$ for each $v$. We call $w$ a \emph{weight function} on $V^2$. Let $\max(w)$ denote the maximum of all the values~$w(u,v)$.

Define the (weighted) degree of a vertex $u$ by setting $\deg_w(u)=\sum_v w(u,v)$. Similarly, define the (weighted) volume of a vertex set $X$ by $\vol_w (X) = \sum_{u \in X} \deg_w(u)$, and (corresponding to~$e(X)$) let $e_w(X)= \tfrac12 \sum_{u,v \in X} w(u,v)$. 

Assume that $w$ is not identically zero, that is $\vol_w(V)=  2 e_w(V)>0$. For a given partition $\cA$ of $V$, define the \emph{modularity score} of $\cA$ on $w$ by 
\begin{eqnarray*}
q_\cA(w)  & = &
\frac{1}{\vol_w(V)}\sum_{A\in \cA} \sum_{u,v \in A} \left(w(u,v) -\frac{\deg_w(u) \deg_w(v)}{\vol_w(V)}\right)\\
& = & q_{\cA}^E(w) - q_{\cA}^D(w)
\end{eqnarray*}
where
\[ q_{\cA}^E(w) = \frac{1}{\vol_w(V)} \sum_{A \in \cA} 2 e_w(A) \;\;\mbox{ and }\;\; q_{\cA}^D(w) =
\frac{1}{\vol_w(V)^2} \sum_{A \in \cA} \vol_w(A)^2.\]
Define the \emph{modularity} of $w$ by $q^*(w)=\max_{\cA} q_{\cA}(w)$.
As in the unweighted case, $0 \leq \q(w)<1$, and we may ignore vertices with degree 0.  If $w$ is identically 0 we set $q_{\cA}(w)=0$ and $\q(w)=0$. If $w$ is $\{0,1\}$-valued then $q_{\cA}(w)$ and $\q(w)$ are the usual modularity score and modularity value respectively for the graph corresponding to $w$. The values $q_{\cA}(w)$ and $\q(w)$ are unchanged under re-scaling $w$, so we may assume that $0 \leq w(u,v) \leq 1$ for each $u,v \in V$; and in this case we call $w$ a \emph{probability weight function}.

Given a weight function $w$ on $V^2$ and $0 < p \leq 1$ we define a random weight function~$w_p$ by considering each edge $uv$ independently and keeping $w(u,v)$ unchanged with probability $p$, and otherwise setting it to 0.  
Also, given a probability weight function $w$, let $G_{w}$ be the random (unweighted) graph obtained by considering each edge $uv$ independently, and including edge $uv$ with probability $w(u,v)$, and otherwise having no edge $uv$. 
For a special case of $G_w$ see Definition~3.1 in~\cite{koshelev2023modularity}. The model has also been considered in network science applications and $G_w$ is referred to as a probabilistic network (to distinguish it from a binary one - in which all edges have probability either 0 or 1 of appearing)~\cite{kaveh2019comparing,poisot2016structure}.

\needspace{3\baselineskip}
\subsection{Weighted underlying to weighted observed graph}
Given a probability weight function $w$ on $V^2$ and a probability $p$, we defined the random weight function $w_p$ above. 
For such weight functions we have results very like Theorems~\ref{thm.obsmod} and~\ref{thm.moddiff}. Note that the theorems below have conditions on the sum of weights $e_w(V)$ being large - and recall that the modularity score of $w$ is unchanged under re-scaling - thus given a general weight function $w$ we may best apply the theorems to a re-scaling $\tilde{w}$ of $w$ by dividing through by $\max(w)$  (the maximum weight of an edge).

\begin{theorem}\label{thm.obsmodw}
Given $b>0$ and $\eps>0$, there exists $c=c(\eps)$ such that the following holds.
Let $0<p \leq 1$, and let the probability weight function $w$ on $V^2$ satisfy $e_w(V)\,p \geq c$.  Then, with probability $\geq 1-\eps$, the random weight function $w_p$ satisfies $\q(w_p) \geq \q(w) - \eps$.  
\end{theorem}

\needspace{6\baselineskip}
\begin{theorem} \label{thm.moddiffw}
Given $\eps>0$, there exists $c=c(\eps)$ such that the following holds.
Let $0<p \leq 1$, and let the probability weight function $w$ on $V^2$ satisfy $e_w(V)\,p \geq c\, |V|$.  Then, with probability~$\geq 1-\eps$, 
\begin{itemize}
    \item[(a)] the random weight function $w_p$ satisfies $|\q(w_p) - \q(w)| <\eps$; and 
    \item[(b)] given any partition $\cA$ of the vertex set, in a linear number of operations (seeing only~$G_p$) the greedy amalgamating algorithm finds a partition $\cA'$ with \mbox{$q_{\cA'}(w) \geq q_{\cA}(w_p) \!-\!\eps$}.
\end{itemize}
\end{theorem}
Observe that Theorem~\ref{thm.obsmod} is the special case of Theorem~\ref{thm.obsmodw} when $w$ is $\{0,1\}$-valued, and similarly Theorem~\ref{thm.moddiff} is a special case of Theorem~\ref{thm.moddiffw}.
In order to prove Theorems~\ref{thm.obsmodw} and~\ref{thm.moddiffw}, we may use almost the same proofs as before. We need a natural minor variant of Lemma~\ref{lem.nosmall2}.
Given a weight function $w$ on $V^2$, call a set $U$ of vertices $\eta$-{\em fat} if $\vol_w(U) \geq \eta\, \vol_w(V)$, and call a partition $\cA$ of $V$ $\eta$-{\em fat} if each part is $\eta$-{\em fat}.  The following lemma is very similar to Lemma~\ref{lem.nosmall2} and may be proved in exactly the same way.
\begin{lemma}[The fattening lemma for weighted graphs] \label{lem.nosmallw}
  For each non-zero weight function $w$ on $V^2\!$, and each \mbox{$0<\eta \leq 1$,} there is an $\eta$-fat partition $\cA$ of $V$ such that $q_{\cA}(w) > \q(w) - 2\eta$.  
  Indeed, given any partition $\cA_0$ of $V$,  using a linear number of operations, by amalgamating parts we can construct an $\eta$-fat partition $\cA$ such that $q_{\cA}(w) > q_{\cA_0}(w) - 2\eta$.
\end{lemma}

\begin{proof}[Proof of Theorem~\ref{thm.obsmodw}.]  The proof is very similar to that of Theorem~\ref{thm.obsmod} so we indicate just a few key steps where there are differences.

As in that proof we may assume that $0< \eps<1$ and $\q(w)\geq \eps$, and we set $\eta=\eps/4$.  By the weighted fattening lemma, Lemma~\ref{lem.nosmallw}, there exists an $\eta$-fat partition $\cA=\{A_1, \ldots, A_k\}$ (where~$k\leq 1/\eta$) such that $q_\cA(w)\geq \q(w)-2\eta\geq 2\eta$. Let $t=(e_w(V)p)^{1/2}$. Then corresponding to~\eqref{eqn.edge}, 
for $0\leq x \leq t$ (and noting that $w(u,v)\leq 1$ for each edge $uv$)
\[\pr\big( \; | e_{w_p}(V)- e_w(V)p| \geq e_w(V)p \cdot x/t\;) \leq 2e^{-x^2/3}. \]
Let $e^{\rm int}_\cA(w)$ denote the sum of `internal' edge weights within the parts of $\cA$. 
Then corresponding to~\eqref{eqn.qE}, for $0< x \leq (3\eta)^{1/2}t$, with probability at least $1-4e^{-x^2/3}$ 
\[
q^E_\cA(w_p) = \frac{e^{\rm int}_\cA(w_p)}{e_{ w_p}(V)} \geq q^E_\cA(w) \frac{ 1-(3\eta)^{-1/2}x/t }{1+x/t}.
\]
For the degree tax, corresponding to~\eqref{eqn.qD} we find the following. For $0 < x \leq (2\eta)^{1/2}t$, with probability at least $1-2(k+1)e^{-x^2/(6b)}$
\[
q^D_\cA( w_p) \leq q^D_\cA(w) \left( \frac{ 1+(2\eta)^{-1/2}x/t }{1-x/t} \right)^2.\] 
Thus for $0<x\leq (2\eta)^{1/2}t$, with probability at least $1-2(k+3)e^{-x^2/6}$  
both the last two displayed inequalities hold. The failure probability is at most $2(4/\eps+3) e^{-x^2/6}$, and we may thus choose $x=x(\eps)$ sufficiently large that the probability is at most $\eps$; and indeed we may take $x=\Theta((\log \eps^{-1})^{1/2})$. 

The rest of the proof follows the non-weighted case by making similar minor adaptations.\end{proof}

\begin{proof}[Proof of Theorem~\ref{thm.moddiffw}] As in the last proof, we may follow the proof of the non-weighted version, Theorem~\ref{thm.moddiff}, with the following adaptations. In place of the fattening lemma, use the weighted version, Lemma~\ref{lem.nosmallw}; replace instances of $G$ and $G_p$ by $w$ and $w_p$ respectively. We may still apply Lemma~\ref{lem.conc} with $0\leq X_j \leq 2$ since all edges have weight at most 1.
\end{proof}

\needspace{3\baselineskip}
\subsection{Weighted underlying to unweighted observed graph}
We will see that the proofs in this subsection follow our proofs of Theorems~\ref{thm.obsmod} and~\ref{thm.moddiff} almost line by line, replacing all instances of $G$ with $w$ and all instances of $G_p$ with $G_w$.

\begin{theorem} \label{thm.obsmod_wb} There exists $c=c(\eps)$ such that the following holds.
Let the probability weight function $w$ on $V^2$ satisfy $e_w(V) \geq c$.  Then with probability at least $1-\eps$ the random graph $G_{w}$ satisfies $\q(G_{w}) > \q(w) - \eps$. 
\end{theorem}

\needspace{3\baselineskip}
\begin{theorem}\label{thm.moddiff_wb}
Given $\eps>0$, there exists $c=c(\eps)$ such that the following holds.
Let the probability weight function $w$ on $V^2$ satisfy $e_w(V) \geq c\, |V|$.  Then with probability at least $1-\eps$, 
\begin{itemize}
    \item[(a)] the random graph $G_{w}$ satisfies $|\q(G_{w}) - \q(w)| <\eps$; and 
    \item[(b)] 
    given any partition $\cA$ of the vertex set, in a linear number of operations (seeing only~$G_p$) the greedy amalgamating algorithm finds a partition $\cA'$ with \mbox{$q_{\cA'}(w) \geq q_{\cA}(G_{w}) \!-\!\eps$}.
\end{itemize}
\end{theorem}

\begin{proof}[Proof of Theorem~\ref{thm.obsmod_wb}] The proof follows that of the non-weighted case, Theorem~\ref{thm.obsmod}, line by line with the following adaptations. In place of the fattening lemma used on the underlying graph, use the weighted version, Lemma~\ref{lem.nosmallw}; and replace instances of $G$ and $G_p$ by $w$ and $G_{w}$ respectively. Note that Lemma~\ref{lem.conc} still applies with $0\leq X_i \leq 2$ - more details are given in the proof of Theorem~\ref{thm.moddiff_wb}. \end{proof}

\begin{proof}[Proof of Theorem~\ref{thm.moddiff_wb}] As in the proof of Theorem~\ref{thm.moddiff}, let $\eps>0$ and let $c> K \eps^{-3}\log \eps^{-1}$ for a sufficiently large constant $K$. We again set $\eta=\eps/9$. Let $w$ be a fixed probability weight function on $V^2$, where $|V|=n$; and assume that $e_w(V)p\geq cn$.

Then the corresponding events $\cB_0, \cB_1, \cB_2$ and $\cE_0$ may all be defined as before, replacing instances of $G$ and $G_p$ with $w$ and $G_{w}$ respectively. Note that in the definition of $\cE_0$, one constructs partition $\cA_0'$ from $\cA_0$ seeing only the observed graph $G_{w}$, and thus using the usual fattening lemma, Lemma~\ref{lem.nosmall2} rather than the weighted one. Then the proof proceeds by noting $\pr(\cB_0)$ is small due to Theorem~\ref{thm.obsmod_wb} and proving the statements corresponding to \eqref{eqn.a} to \eqref{eqn.B2}.

Corresponding to the proof of \eqref{eqn.a} only the usual substitutions of $G$ and $G_p$ by $w$ and $G_{w}$ are required. For the observation after that a small change is needed. Notice that since $\max(w) \leq 1$ we have that $v(G)^2/2 \geq e_w(V)$ and $e_w(V) \geq c v(G)$ by assumption. Thus $v(G)\geq 2c$ as in the earlier proof.

Corresponding to the proof of \eqref{eqn.B1} there are a few minor changes. Let $A\subseteq V$ have $\vol_w(A)<\frac{\eta}{2}\vol_w(V)$. Suppose the (unordered) pairs $u,v$ of vertices in $A$ with $w(u,v)>0$ are labelled $u_1v_1, \ldots, u_jv_j$ (some vertices may be repeated), and the pairs of vertices with exactly one endpoint in $A$ and positive edge weight are labelled $u_{j+1}v_{j+1},\ldots, u_kv_k$. For $i=1, \ldots, j$ let $X_i$ be 2 if edge $u_iv_i$ is in~$G_{w}$ and let $X_i$ be 0 otherwise; and for $i=j+1,\ldots,k$ let $X_i=1$ if edge $u_iv_i$ is in $G_{w}$ and let $X_i$ be 0 otherwise (so the random variable $X_i$ satisfies $0\leq X_i \leq 2$ if $i\leq j$ and $0\leq X_i \leq 1$ if $i>j$, and is non-zero with probability $w(u_i,v_i)$.)
Corresponding to the original proof, $\vol_{G_{w}}(A) =\sum_i X_i$ and $\vol_w(A)=\E[\sum_i X_i]$, and the rest of the proof corresponding to~\eqref{eqn.B1} follows in the same manner -- note that Lemma~\ref{lem.conc} (with~$b=2$) still applies. 

Corresponding to the proof of Claim~\eqref{eqn.eintsmall2}, let $\cB_5$ be the event that, for some partition $\cA \in \cQ$ such that $\eint(w) < \tfrac{\eta}2 e_w(V)$, we have $\eint(G_{w}) \geq \eta e_w(V)$. We note that $\eint(G_{w})$ is stochastically at most a sum of Bernoulli random variables, such that the mean (of the sum) is $\eta e_w(V)/2$ so we may apply Lemma~\ref{lem.conc} (with $b=1$). The rest of the proof corresponding to Claim~\eqref{eqn.eintsmall2} and indeed the entire proof continues with similar adaptations.
\end{proof}

\needspace{3\baselineskip}
\subsection{Application : stochastic block model}

In this subsection we show that Theorem~\ref{thm.SBMk} on the modularity value $\q(G_{n,k,p,q})$ of the stochastic block model follows quickly from Theorem~\ref{thm.moddiff_wb} (a weighted version of Theorem~\ref{thm.moddiff}) and the deterministic result Lemma~\ref{lem.weightedSBMk}.
For $k \geq 2$ and $0 \leq q \leq p$ let
\[ q(k,p,q) = \frac{(p-q) \, (1 - 1/k)}{ p + (k-1)q}.\]
Since rescaling a weight function $w$ does not change $\q(w)$, for simplicity we set $\alpha=1$ in the following lemma. 

\begin{lemma}\label{lem.weightedSBMk}
Let $k \in \NN, k\geq 2$.  For $n \in \NN$ let $V_1 \cup \cdots \cup V_k$ be a partition of $V=[n]$ where
$\lfloor n/k \rfloor \leq |V_i| \leq \lceil n/k \rceil$. Let $\alpha=1$ and $0\leq \beta=\beta(n) \leq \alpha$, and let $w=w(n,k,\alpha, \beta)$  be the weight function on vertex set $V$ with $w_{uv}=\alpha$ if $u$ and $v$ are in the same block $V_i$ and with $w_{uv}=\beta$ otherwise. Then 
$q_{\cA_0}(w) = q(k,\alpha,\beta) +o(1)$ for the planted partition $\cA_0$, and $\q(w) = q(k,\alpha,\beta) +o(1)$.
\end{lemma}

Theorem~4.2 in the recent paper by Koshelev~\cite{koshelev2023modularity} concerns the same weighted graph as above and gives an upper bound on $\q(w)$ without the factor $(1-1/k)$, i.e.\  $\q(w)\leq (\alpha-\beta)/(\alpha+ (k-1)\beta)+o(1)$. The proof in~\cite{koshelev2023modularity} proceeds by calculating the eigenvalues of the weighted adjacency matrix of $w$. The proof below of Lemma~\ref{lem.weightedSBMk} involves a simple weighted notion, $f_w(A)$, of `per-unit-modularity' as used in~\cite{w1hard,modexpansion}. 
\begin{proof}
It will be  straightforward to show that the planted partition $\cA_0$ has modularity score as claimed: the main part of the proof is to show that $\q(w)$ is at most this value.

First we show that we may write the modularity of the weight function $w$ as a weighted sum of a function $f_w(A)$ on vertex sets $A$ - see~\eqref{eq.fwdef}. The proof of the upper bound will proceed by bounding the maximum value of $f_w(A)$ over vertex sets $A$. By definition, for any partition $\cA$ of~$V$
\begin{eqnarray*}
q_\cA(w)  & = & \frac{1}{\vol_w(V)}  \sum_{A \in \cA} \left( 2 e_w(A) - \frac{\vol_w(A)^2}{\vol_w(V)} \right) =  \sum_{A \in \cA} \frac{\vol_w(A)}{\vol_w(V)} f_w(A)
\end{eqnarray*}
where 
\begin{equation}\label{eq.fwdef} f_w(A) = \frac{2 e_w(A)}{\vol_w(A)} - \frac{\vol_w(A)}{\vol_w(V)} .
\end{equation}

Let $\eta>0$ and let $\cA$ be an $\eta$-fat partition of $V$.  By Lemma~\ref{lem.nosmallw} it will suffice for us to show that $q_\cA(w) \leq q(k,\alpha,\beta) +o(1)$; and since $q_\cA(w)$ is a weighted average of the values $f_w(A)$, it will suffice to show that $f_w(A) \leq q(k,\alpha,\beta) +o(1)$ for every $A \subseteq V$ with $\vol_w(A) \geq \eta \, \vol_w(V)$.  Fix such a set $A$.  Note that $\vol_w(V) = \Theta(n^2)$, and so also  $\vol_w(A) = \Theta(n^2)$.

Define $\delta_i$ such that $|A\cap V_i|=\delta_i |V_i|$, and note that $0 \leq \delta_i \leq 1$ with $\sum_i \delta_i > 0$.
Observe that the weighted complete graph on $[n]$ is approximately regular with weighted degree 
\[\deg_w(u)=(\alpha+ (k-1)\beta)\, n/k +O(1)\] for each vertex $u$ (where the $O(1)$ error term is in $(-2\alpha + \beta, \beta) \subseteq (-2,1)$).
 Thus $\vol_w(V)=(\alpha+(k-1)\beta)n^2/k+O(n)$ and $\vol_w(A)=(\alpha+(k-1)\beta)(\sum_i \delta_i) n^{2}/k^2 + O(n)$.
Also
\begin{eqnarray*}
2e_w(A)
& = & \alpha \frac{n^2}{k^2} \big(\sum_i \delta_i^2\big) + \beta \frac{n^2}{k^2} \big( \sum_{i \neq j} \delta_i \delta_j \big) + O(n)\\
& = & 
\frac{n^2}{k^2} \left( (\alpha - \beta) (\sum_i \delta_i^2) + \beta(\sum_i \delta_i)^2 \right) +O(n)
\end{eqnarray*}
since $\sum_{i \neq j} \delta_i\delta_j = (\sum_i \delta_i)^2 - \sum_i \delta_i^2$.
Thus 
\begin{eqnarray*}
    f_w(A) 
    & = & 
\frac{(\alpha-\beta) (\sum_i \delta_i^2) + \beta (\sum_i\delta_i)^2}{(\alpha+(k-1)\beta)(\sum_i \delta_i)}
- \frac{ (\alpha +(k-1)\beta)(\sum_i \delta_i)}{k(\alpha+(k-1)\beta)} + O\Big(1/n\Big)\\
&=&    
\frac{(\alpha-\beta) (\sum_i \delta_i^2) +  (\beta - \frac1{k}(\alpha+(k-1)\beta) (\sum_i \delta_i)^2}{
(\alpha+(k-1)\beta)(\sum_i \delta_i)} + O\Big(1/n\Big)\\
&=& 
\frac{(\alpha-\beta) ((\sum_i \delta_i^2) - \frac1{k}(\sum_i \delta_i)^2)}{ (\alpha + (k-1)\beta) (\sum_i \delta_i)} + O\Big(1/n\Big).
\end{eqnarray*}
Now, for any $s>0$, given that $\sum_i \delta_i=s$ 
\[ f_w(A)= \frac{(\alpha-\beta) \, ((\sum_i \delta_i^2)/s  - s/k)}{ \alpha + (k-1)\beta} + O\Big(1/n\Big).\]
We now show $(\sum_i \delta_i^2)/s  - s/k\leq 1 - 1/k$. If some $\delta_i=1$ and the other $\delta_j$ are zero (in other words, if $A$ is $V_i$) then $s=1$ and $(\sum_i \delta_i^2)/s  - s/k = 1-1/k$.
If $s \leq 1$ then
$(\sum_i \delta_i^2)/s  - s/k \leq s-s/k \leq 1-1/k$.
Also, suppose that $s>1$ say $s=a+x$ where $a \in \NN$ and $0 \leq x<1$.  Then $\sum_i \delta_i^2 \leq a+x^2$, and
$(\sum_i \delta_i^2)/s  - s/k \leq \frac{a+x^2}{a+x} -\frac{s}{k} < 1-1/k$.
Thus
\[ f_w(A) \leq q(k,\alpha,\beta) + O\Big(1/n\Big)\]
(and the upper bound is achieved when $A$ is a block $V_i$).
But, as we noted earlier, $q_\cA(w)$ is a weighted average of the values $f_w(A)$ for $A \in \cA$, and so
\[ q_\cA(w) \leq q(k,\alpha,\beta) + O\Big(1/n\Big).\]
Hence
\[ \q(w) \leq q(k,\alpha,\beta) + o(1),\]
and we have the upper bound claimed.

Now take $\cA$ as the planted partition $\{V_1,\ldots,V_k\}$, and note that  $f_w(V_i)= q(k,\alpha,\beta) + O(1/n)$ for each $i$.  Now since $q_\cA(w)$ is a weighted average of the values $f_w(V_i)$ we see that
\[ q_\cA(w) =  q(k, \alpha,\beta) + O(1/n) , \]
and we are done.
\end{proof}

\begin{proof}[Proof of Theorem~\ref{thm.SBMk}.] 
Let $\alpha=1$ and $\beta=\rho$ (so $0 \leq \beta \leq \alpha$).  Let $\eps>0$. Let $w=w(n,k,\alpha,\beta)$ as in Lemma~\ref{lem.weightedSBMk}, and let $\hat{w}= p\, w$.  Then $G_{n,k,p,q}$ is $G_{\hat{w}}$, and whp $|\q(G_{\hat{w}}) - \q(\hat{w})| \leq \eps$ by Theorem~\ref{thm.moddiff_wb}
since with $V=[n]$ we have
$e_{\hat{w}}(V)/n \geq \frac{np}{2k} +O(1) \to \infty$ as $n \to \infty$. 
But $\q(\hat{w}) = \q(w)$ and $q-\cA(\hat{w}) = q_\cA(w)$ for the planted partition $\cA$, so by Lemma~\ref{lem.weightedSBMk}
\[ q_\cA(G_{n,k,p,q}) = q(k,\alpha,\beta) +o(1)  \;\;\; \mbox{whp}\]
and
\[ \q(G_{n,k,p,q}) = q(k,\alpha,\beta) +o(1)  \;\;\; \mbox{whp},\]
and we are done.
\end{proof}

There is a version of the stochastic block model in which vertices are assigned to blocks independently and uniformly at random.
Consider the variant of Theorem~\ref{thm.SBMk} using this version of the stochastic block model.
Note that the part sizes will whp be $n/k \pm n^{1/2}\log n$. Thus by a coupling argument the edge set will differ by $o(m)$ edges whp from the original model; and by the robustness result, Lemma~\ref{lem.robustness}, we find that whp the modularity values $q_\cA$ and $\q$ are both $q(k,\alpha,\beta) + o(1)$ as before.

\needspace{6\baselineskip}
\section{Concluding remarks}\label{sec.concl}

\noindent\emph{Three phases of modularity}\\
As mentioned in Section~\ref{subsec.relation}, the modularity of the binomial random graph $\Gnp$ exhibits three phases~\cite{ERmod}. We restate the result as a sampling one. Recall that $\q(K_n)=0$~\cite{nphard}. Let $G=K_n$ and consider $n \to \infty$. Then in the sparse case (where $e(G)p\rightarrow \infty$ and $e(G)p/n\leq \frac12 + o(1)$) we have $\q(G_p)$ near~1 whp, in the dense case (where $e(G)p/n \rightarrow \infty$) $\q(G_p)$ is near~$\q(G)=0$ whp, and in between (where $c_1 \leq e(G)p/n \leq c_2$ for constants $\frac12 <c_1<c_2$) $\q(G_p)$ is bounded away from $\q(G)=0$ and $1$ whp. The question below asks if we may extend this three phase behaviour from complete graphs to a larger class of underlying graphs. Let us restrict our attention to graphs without isolated vertices.

\needspace{3\baselineskip}
\begin{question} For which classes $\mathcal{H}$ of graphs $H$ do we have parts (i) and (ii) of the following three phase result (where we write $n$ for $v(H)$)?
\begin{itemize}
\item[(i)] If $e(H)p\rightarrow \infty$ and $e(H)p/n \leq \frac12$ then $\q(H_p)=1+o(1)$ whp.
\item[(ii)]  
There exist $\eps>0$ and $ 0< c_1 < c_2$ such that if $c_1 \leq e(H)p/n \leq c_2$ then $\q(H)+\eps < \q(H_p) < 1-\eps$ whp.
\item[(iii)] If $e(H)p/n \rightarrow \infty$ then $\q(H_p)=\q(H)+o(1)$ whp.
\end{itemize}

\end{question}
Firstly, observe that part (iii) is implied by Theorem~\ref{thm.moddiff}: thus the open question is for which classes of graphs we have parts (i) and (ii). Secondly, we can only get three genuine phases if $e(H)/n$ is unbounded. We have already noted that the three phase result holds for complete graphs, and by double exposure it must hold also for the random graph $G_{n,p}$ where $np \to \infty$ as $n \to \infty$.

Part (i) is false for complete bipartite graphs $K_{k,n}$ for any fixed $k$, since $\q(G) \leq 1-1/k$ for any subgraph $G$ of $K_{k,n}$ (since $G$ has at most $k$ components, ignoring isolated vertices).  Similarly, if $t$ and $k_1,\ldots,k_t$ are fixed and $H$ has components $K_{k_1,n_1},\ldots, K_{k_t,n_t}$ then part (i) is false. Also, for (ii) to hold, we clearly must have $\q(H)$ bounded away from~1.

\medskip

\needspace{3\baselineskip}
\noindent\emph{Random false positives}\\
In our model, the observed graph $G_p$ has random false negatives, that is, there are are edges in $G$ which are non-edges in $G_p$, but no false positives. Theorem~\ref{thm.moddiff} shows that the modularity value is robust to random false negatives: so long as the observed graph has high enough expected degree the modularity value of the observed graph is close to the modularity of the underlying graph. (For an underlying graph with $\Theta(n^2)$ edges we may have a $1-\Theta(1/n)$ chance of not seeing each edge.) However, perhaps random false positives may cause a large increase or decrease to the modularity value, possibly as much as adversarially added false positives.

Suppose that we are given all the edges in $G$, but we may falsely think that some  non-edges -- the false positives -- are also edges of $G$. (We have not allowed false positives until now.) What can we say about modularity?
Let us consider graphs $G_n$ with $n$ vertices and $m=m(n)$ edges which we see fully and correctly; and suppose that there are $\tilde{m} = \tilde{m}(n)$ false positives randomly chosen from the non-edges of the graph, so the observed graph~$G'_n$ has $m+\tilde{m}$ edges.

If $\tilde{m}$ is much smaller than $m$ then (unsurprisingly) we have no difficulties: by Lemma~\ref{lem.robustness}, if $\tilde{m} \leq (\eps/2) m$ then $|\q(G_n)-\q(G'_n)| \leq \eps$ so $\q(G'_n)$ is a good estimate of $\q(G_n)$.  At the other extreme, if $\tilde{m}$ is much larger than $m$ then there is little point in using $\q(G'_n)$ as an estimate for~$\q(G_n)$.
For by Lemma~\ref{lem.robustness}, if $\tilde{m} \geq (2/\eps) m$ and we denote by $\tilde{G}_n$ the graph formed just from the false positive edges, then $| \q(G'_n) - \q(\tilde{G}_n)| \leq \eps$; so that $\q(G'_n)$ is essentially determined by~$\q(\tilde{G}_n)$, whatever the value of $\q(G_n)$.

The interest is thus in the balanced case, when $\tilde{m} \sim \delta m$ for some constant $\delta$. We will see in Example~\ref{ex.add_random_edges_to_increase_q} that sometimes randomly adding false positives may \emph{increase} the modularity nearly as much as adversarially choosing edges to add. Part (a) of the robustness lemma, Lemma~\ref{lem.robustness}, gives the following `adversarial' bound. If the graph $G$ has $m$ edges, and $G'$ is any graph obtained from $G$ by adding $\delta m$ edges, then for any vertex partition~$\cA$,
\begin{equation}\label{eq.add_edges_adversarial}
|\q(G)-\q(G')|,\, |q_\cA(G)-q_\cA(G')| \;\leq \frac{2\delta}{1+\delta}.
\end{equation}
\needspace{3\baselineskip}
\begin{example}\label{ex.add_random_edges_to_increase_q}
Let $k=k(n) \in \NN$ and suppose that $2 \leq k \ll n^{1/2}$.  Let~$G_n$ consist of a star on $[k]$ together with $n-k$ isolated vertices; and note that $m=e(G_n)=k-1$ and $\q(G_n)=0$. Form~$G'_n$ by adding~$\delta m$ new edges uniformly at random.  Then the difference in modularity values is approximately $2\delta/(1+\delta)$, matching the adversarial bound~\eqref{eq.add_edges_adversarial}.

To see this we may check that whp $G'_n$ consists of the star together with $\delta m$ isolated edges (and $n-k-2 \delta m$ isolated vertices - which we may ignore). But for such a graph $H$, an optimal partition~$\cA$ consists of one part $[k]$ and $\delta m$ parts of size 2 corresponding to the isolated edges - since the parts in an optimal partition must induce connected subgraphs, and connected components which themselves have modularity zero must not be split in an optimal partition~\cite{modexpansion}. (Note that the modularity values of a star and of a single edge are zero~\cite{nphard}). Now, this partition captures all edges and so \[q_{\cA}(H) = 1 - \frac{(2m)^2 + \delta m \cdot 2^2}{(2m(1+\delta))^2}= \frac{2\delta+\delta^2}{(1+\delta)^2} - \frac{\delta}{(1+\delta)^2m} = \frac{2\delta}{1+\delta} - \frac{\delta^2}{(1+\delta)^2} + O(\frac{1}{m}).
    \]  
Hence whp 
\[ \q(G'_n)-\q(G_n) = \q(G'_n) =\frac{2\delta}{1+\delta} - \frac{\delta^2}{(1+\delta)^2} + O(\frac{1}{m})\,,\]
which is close to the adversarial (worst case) bound in~\eqref{eq.add_edges_adversarial} when $\delta$ is small.
\end{example}
However, how much we may \emph{decrease} the modularity by adding random false positives is open. Indeed, it is open how much we may decrease the modularity by adversarially adding edges: note that  Example~\ref{ex.robustness} which showed tightness of Lemma~\ref{lem.robustness} only showed tightness for how much the modularity may \emph{increase} by adding edges.
\begin{question}\label{q.add_edges_decrease_mod}
Let $G$ be a graph with $m$ edges, let $G'$ be obtained from $G$ by randomly adding $\delta m$ edges, and let $G''$ be obtained from $G$ by adding a set of $\delta m$ edges which maximises $\q(G)-\q(G'')$. What upper bounds do we have for $\q(G) - \q(G')$ or for $\q(G)-\q(G'')$? Can we beat the upper bound $2\delta/(1+\delta)$ significantly?
\end{question}

\bibliography{articles}
\appendix

\newpage
\section{Simulations for larger underlying graph}\label{sec.simsBIG}
\begin{figure}[h!]
\begin{picture}(170,550) 
	\put(20,289){ 
		\includegraphics[width=0.8\textwidth]{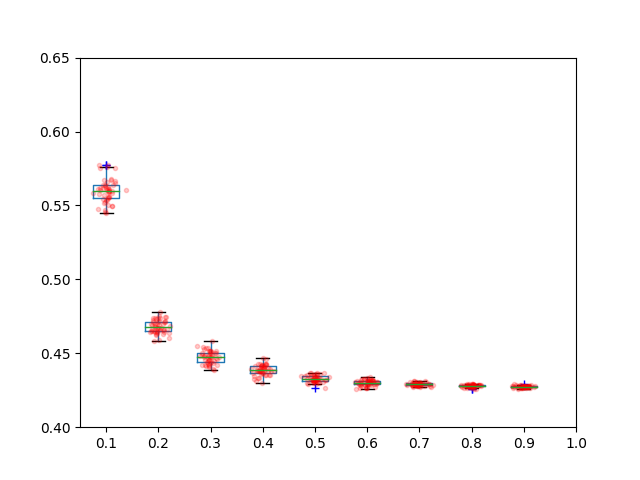}
		\put(-380,135){\rotatebox{90}{$\tilde{q}(G_p)$}} 
	    \put(-300,259){Estimated modularity of sampled graph} 
	    \put(-200,7){$p$}
		}
	\put(20,-6){
		\includegraphics[width=0.8\textwidth]{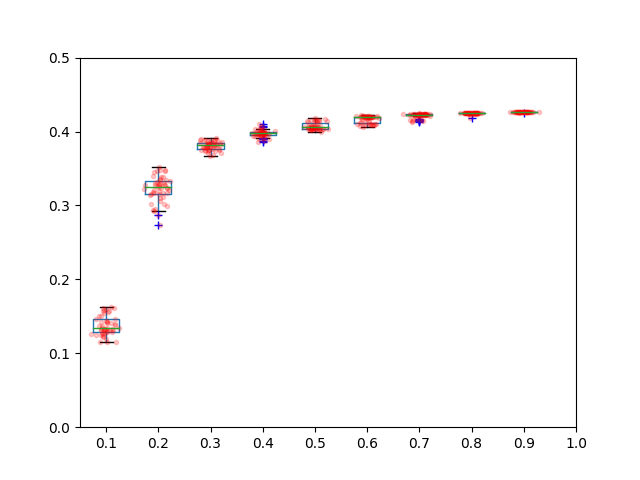}
		\put(-380,135){\rotatebox{90}{$q_{\tilde{\cA}(G_p)}(G)$}} 
	    \put(-300,273){Modularity score of underlying graph using}
	    \put(-300,259){partition estimated from sampled graph}
	    \put(-200,7){$p$}
		}
\end{picture}
\caption{
Simulation results. 
The US political blogs network~\cite{adamic2005political} with 1490~vertices and 16718~edges was taken to be the underlying graph $G$, with the directed links between blogs encoded as  undirected edges. The simulations were run as described in Figure~\ref{fig.sims} on page~\pageref{fig.sims} with the exception that in the lower plot we plot the modularity score of~$\tilde{\cA}(G_p)$, rather than the score of the $\eta$-fattened partition~$\tilde{\cA}'(G_p)$. \vspace{-10mm}
} \label{fig.simsBIG}
\end{figure}

\end{document}